\documentclass[11pt]{article}
\usepackage[margin=1.15in]{geometry}
\usepackage{placeins}
\usepackage{amsmath,amssymb,amsthm}
\usepackage{tikz}
\usepackage{longtable}
\usepackage{colortbl}
\usepackage{afterpage}
\usepackage[hidelinks]{hyperref}

\newtheorem{theorem}{Theorem}

\newcommand{\dens}{\mathrm{d}}
\title{Identifying codes of Archimedean grids}

\author{Rudini M. Sampaio\\[2pt]
\small Dept. Computa\c{c}\~ao, Universidade Federal do Cear\'a, Fortaleza, Brazil\\
\small\texttt{rudini@dc.ufc.br}}

\date{}

\begin{document}
\maketitle

\begin{abstract}
An identifying code of a graph is a dominating set $C$ such that the sets $N[v]\cap C$ are pairwise distinct. The minimum density of an identifying code has been studied extensively for several infinite grids. It is known to be $1/4$ for the triangular grid and $7/20$ for the square grid, whereas for the hexagonal grid the problem is still open, between $5/12$ and $53/126$. These three grids are the regular members of the family of the eleven Archimedean grids. In this paper, we start a systematic study of identifying codes in this family. We determine the minimum density for five of the nine Archimedean grids for which it was not known: it is $2/3$ for the truncated hexagonal grid $3.12.12$, $3/7$ for the truncated square grid $4.8.8$, and $1/3$ for the trihexagonal (kagome) grid $3.6.3.6$, the elongated triangular grid $3^3.4^2$ and the snub hexagonal grid $3^4.6$. All lower bounds are proved by the discharging method. For the three remaining non-regular Archimedean grids, we also obtain some bounds: the minimum density lies between $3/7$ and $13/30$ for the truncated trihexagonal grid $4.6.12$, between $1/3$ and $11/30$ for the rhombitrihexagonal grid $3.4.6.4$, and between $2/7$ and $9/28$ for the snub square grid $3^2.4.3.4$. As a by-product, we prove that every identifying code of a cubic graph with the slow growth property in which every vertex lies on exactly one triangle (resp. $4$-cycle) has density at least $2/3$ (resp. $3/7$).
\end{abstract}



\section{Introduction}

Let $G=(V,E)$ be a graph and let $N[v]$ denote the closed neighborhood of a vertex $v$. A set $C\subseteq V$ is an \emph{identifying code} (\emph{idcode} for short) of $G$ if the sets $I(v)=N[v]\cap C$ are nonempty and pairwise distinct for $v\in V$. Identifying codes were introduced by Karpovsky, Chakrabarty and Levitin~\cite{KCL98}, motivated by fault diagnosis in multiprocessor networks, and have been studied extensively since; see the survey~\cite{HJL24} and the references therein.

For infinite grids, the relevant parameter is the minimum density of an identifying code. Among the three regular planar grids, it is known exactly for the triangular grid~\cite{KCL98} (equal to $1/4$), and for the square grid~\cite{BHL05,CGHLMPZ99} (equal to $7/20$). The hexagonal grid is a well-known open case: the best known lower bound is $5/12$~\cite{CuY13} (see also~\cite{CHLZ00,CY09,Sto14}) and the best known upper bound is $53/126$~\cite{ST24}. The minimum density is also known for the king grid, where it equals $2/9$~\cite{CHL02,CHLZ01}, and grids with finitely many rows were studied in~\cite{DHS17,DHS18,SSW24}.

The triangular, square and hexagonal grids are the three regular members of the family of the eleven \emph{Archimedean grids}, the graphs of the edge-to-edge tilings of the plane by regular polygons whose symmetry group acts transitively on the vertices (Section~2). Domination-type parameters of the Archimedean grids have been studied, for instance perfect domination in~\cite{ZWM22}. However, to the best of our knowledge, among the eight non-regular Archimedean grids the only one for which identifying codes have been considered is the snub hexagonal grid $3^4.6$: an undergraduate thesis~\cite{Alb11} proved that its minimum density lies between $3/10$ and $1/3$. Thus the minimum density of an identifying code was known for only two of the eleven Archimedean grids.

In this paper we start a systematic study of identifying codes of the Archimedean grids: we determine the minimum density for five of the nine remaining ones, and we give bounds for the other three.

\begin{theorem}\label{thm:main}
The minimum density of an identifying code is
\begin{itemize}
\item[(a)] $1/3$ for the trihexagonal (kagome) grid $3.6.3.6$;
\item[(b)] $2/3$ for the truncated hexagonal grid $3.12.12$;
\item[(c)] $3/7$ for the truncated square grid $4.8.8$;
\item[(d)] $1/3$ for the elongated triangular grid $3^3.4^2$;
\item[(e)] $1/3$ for the snub hexagonal grid $3^4.6$;
\item[(f)] between $3/7$ and $13/30$ for the truncated trihexagonal grid $4.6.12$;
\item[(g)] between $1/3$ and $11/30$ for the rhombitrihexagonal grid $3.4.6.4$;
\item[(h)] between $2/7$ and $9/28$ for the snub square grid $3^2.4.3.4$.
\end{itemize}
\end{theorem}

In each case, the upper bound is given by an explicit periodic code, and the lower bound is proved by the discharging method, in the framework for infinite graphs with the slow growth property developed in~\cite{SSW24}. The lower bounds have very different levels of difficulty. For the kagome grid, the general lower bound $2/(\Delta+2)$ of~\cite{SSW24} (which comes from the simplest discharging rule, the share argument) already gives the exact value. For the truncated hexagonal grid, the general bound is only $2/5$, but the value $2/3$ follows from a very simple discharging rule, which also gives the lower bound $2/3$ for every cubic graph in which every vertex lies on exactly one triangle. For the truncated square grid, the proof is a short case analysis, which also gives the lower bound $3/7$ for every cubic graph in which every vertex lies on exactly one $4$-cycle. For the elongated triangular grid, the case analysis is longer, but still done by hand. Finally, for the snub hexagonal grid, the discharging rules and the case analysis were found and checked with the help of a computer, and the case analysis is given in full as a certificate whose rows can be checked by hand. A program that checks the certificate of Section~\ref{sec:snub} is available at~\cite{rudini26idc}.

The paper is organized as follows. In Section~2 we recall the Archimedean grids, the notion of density and the discharging method for infinite graphs. Theorem~\ref{thm:main} is the union of Theorems~\ref{thm:kagome}, \ref{thm:trhex}, \ref{thm:tsq}, \ref{thm:elong}, \ref{thm:snub} and~\ref{thm:remaining}, which are stated and proved in Sections~\ref{sec:kagome}, \ref{sec:trhex}, \ref{sec:tsq}, \ref{sec:elong}, \ref{sec:snub} and~\ref{sec:remaining}, respectively. For the three grids of Section~\ref{sec:remaining}, the upper bounds come from the best periodic codes we found with a computer search. Table~\ref{tab:archimedean} summarizes what is known about the eleven Archimedean grids.

\section{Preliminaries}

A tiling of the plane by polygons is \emph{edge-to-edge} if any two tiles are disjoint or meet in a common vertex or a common edge. An \emph{Archimedean tiling} is an edge-to-edge tiling of the plane by regular polygons whose symmetry group acts transitively on its vertices. In such a tiling, the cyclic sequence of the numbers of sides of the polygons around a vertex, its \emph{vertex configuration}, is the same at every vertex up to rotation and reflection; it is written as a sequence separated by dots, with exponents denoting repetition, so that $4.8.8$ means a square and two octagons and $3^3.4^2=3.3.3.4.4$ means three triangles followed by two squares. There are exactly eleven Archimedean tilings~\cite[Section~2.1]{GS87}, including the three regular tilings by triangles, squares and hexagons. The \emph{Archimedean grid} of an Archimedean tiling is the infinite graph formed by its vertices and edges. It is planar, vertex-transitive, and $\Delta$-regular, where $\Delta$ is the number of polygons around a vertex. The eleven Archimedean grids are listed in Table~\ref{tab:archimedean} and shown in Figure~\ref{fig:archimedean}, together with what is known about the minimum density of their idcodes.

All Archimedean tilings except the elongated triangular tiling $3^3.4^2$ can be obtained from the regular tilings by Wythoff's construction~\cite{Cox73}, which uses five operations: truncation, rectification, cantellation, omnitruncation and the snub operation. Truncation gives $3.12.12$ and $4.8.8$, rectification gives $3.6.3.6$, cantellation gives $3.4.6.4$, omnitruncation gives $4.6.12$ and $4.8.8$, and the snub operation gives $3^4.6$ and $3^2.4.3.4$.

Truncation, in particular, can be defined for every graph: the \emph{truncation} $\operatorname{tr}(H)$ of a graph $H$ is obtained by replacing every vertex $v$ of degree $d$ by a cycle of length $d$, each vertex of which receives one of the edges of $v$; when $H$ is a plane graph, the edges are taken in the cyclic order around $v$. Thus $\operatorname{tr}(H)$ is cubic, and every vertex of $\operatorname{tr}(H)$ lies on exactly one of these new cycles. For instance, $3.12.12=\operatorname{tr}(6^3)$, $4.8.8=\operatorname{tr}(4^4)$ and $4.6.12=\operatorname{tr}(3.6.3.6)$.

\begin{table}[t]
\centering
\small
\begin{tabular}{llcl}
\hline
\rowcolor{gray!25}Configuration & Name & $\Delta$ & Minimum density \\
\hline
$3^6$ & triangular & $6$ & $1/4$~\cite{KCL98} \\
$4^4$ & square & $4$ & $7/20$~\cite{CGHLMPZ99,BHL05} \\
$6^3$ & hexagonal & $3$ & $5/12$~\cite{CuY13} to $53/126$~\cite{ST24} \\
\hline
$3.6.3.6$ & trihexagonal (kagome) & $4$ & $1/3$ (Thm.~\ref{thm:kagome}) \\
$3.12.12$ & truncated hexagonal & $3$ & $2/3$ (Thm.~\ref{thm:trhex}) \\
$4.8.8$ & truncated square & $3$ & $3/7$ (Thm.~\ref{thm:tsq}) \\
$3^3.4^2$ & elongated triangular & $5$ & $1/3$ (Thm.~\ref{thm:elong}) \\
$3^4.6$ & snub hexagonal & $5$ & $1/3$ (Thm.~\ref{thm:snub}) \\
\hline
$4.6.12$ & truncated trihexagonal & $3$ & $3/7$ (Thm.~\ref{thm:4cycle}) to $13/30$ (Thm.~\ref{thm:remaining}) \\
$3.4.6.4$ & rhombitrihexagonal & $4$ & $1/3$~\cite{SSW24} to $11/30$ (Thm.~\ref{thm:remaining}) \\
$3^2.4.3.4$ & snub square & $5$ & $2/7$~\cite{SSW24} to $9/28$ (Thm.~\ref{thm:remaining}) \\
\hline
\end{tabular}
\caption{The eleven Archimedean grids, their degree $\Delta$, and what is known about the minimum density of an idcode.}
\label{tab:archimedean}
\end{table}

\begin{figure}[t]
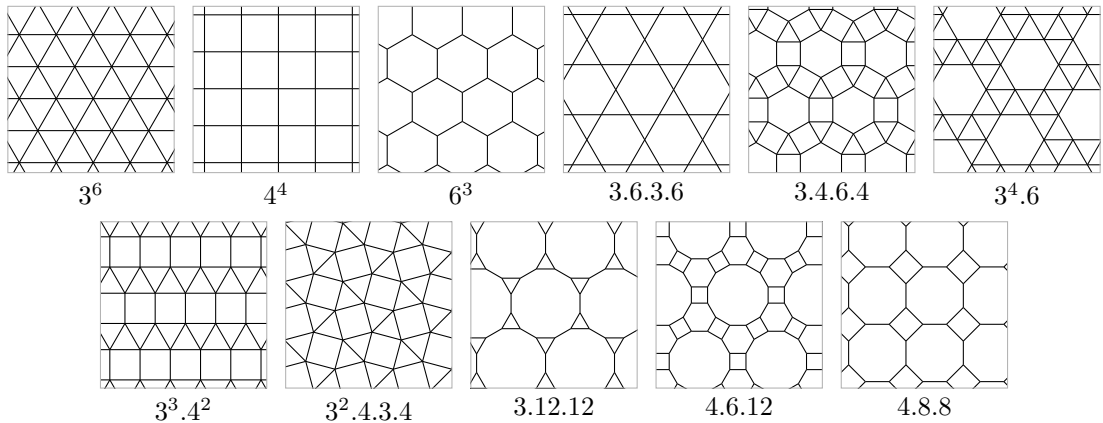

\centering
%
\caption{The 11 Archimedean grids, labeled by their vertex configurations.}
\label{fig:archimedean}
\end{figure}

We define the density of a set of vertices of an infinite graph as in~\cite{SSW24}. For a vertex $w$, let $N_r[w]$ be the set of vertices at distance at most $r$ from $w$. The density of $C\subseteq V$ is $\dens(C)=\inf_{w}\dens_w(C)$, where $\dens_w(C)=\limsup_{r\to\infty}|C\cap N_r[w]|/|N_r[w]|$. In an arbitrary infinite graph this definition must be handled with care. For instance, Slater~\cite{Sla02} stated in 2002 that every locating-dominating set of a countably infinite $d$-regular graph has density at least $2/(d+3)$. As shown in~\cite{SSW24}, this is false for the infinite $3$-regular tree, which has a locating-dominating set of density at most $5/16<1/3$; in this tree, the value $\dens_w(C)$ may moreover depend on the vertex $w$. The property that makes density arguments valid is the following one.

A connected graph has the \emph{slow growth property}~\cite{SSW24} if it has a vertex $s$ such that
\[
\lim_{r\to\infty}\frac{|N_{r+1}[s]|}{|N_r[s]|}=1.
\]
 If this holds for one vertex $s$, then $\lim_{r\to\infty}|N_{r+t}[w]|/|N_r[w]|=1$ for every vertex $w$ and every integer $t$, and $\dens_w(C)=\dens_s(C)$ for every vertex $w$, so that the density of $C$ may be computed from an arbitrary vertex~\cite{SSW24}. Since $|N_r[w]|$ grows quadratically in $r$ in every Archimedean grid, all of them have the slow growth property.

A graph admits an idcode if and only if it is \emph{twin-free}, that is, $N[u]\neq N[v]$ for all distinct vertices $u,v$. It is easy to check that the eleven Archimedean grids are twin-free: two adjacent vertices have at most two common neighbors, and at most one in the cubic grids. The \emph{minimum density} of an idcode of a twin-free graph $G$ is the infimum of $\dens(C)$ over all idcodes $C$ of $G$; for the five grids whose minimum density we determine, it is attained by periodic codes.

All our lower bounds are proved by the \emph{discharging method}, which we explain now. In a \emph{charging phase}, a charge, which is a rational number, is assigned to every vertex; in a \emph{discharging phase}, charge is moved between vertices according to \emph{discharging rules}, which preserve the total charge. The rules are designed so that, at the end, every vertex has charge at least some value $q$. In a finite graph, if every codeword starts with charge $1$ and all other vertices with charge $0$, this immediately gives $|C|\ge q|V|$. In an infinite graph, the total charge is infinite and one has to argue on balls $N_r[w]$, whose boundary may receive charge from outside or send charge outside. The slow growth property is exactly what makes the contribution of the boundary negligible, provided that charge only travels a bounded distance. This is the content of the following theorem.

\begin{theorem}[Discharging Method~{\cite[Theorem~3.3]{SSW24}}]\label{thm:discharging}
Let $G$ be an infinite graph with bounded maximum degree which satisfies the slow growth property, and let $C\subseteq V(G)$. Suppose that the discharging method is applied to $G$ as follows. In the charging phase, charge $1$ is assigned to each vertex of $C$ and charge $0$ to the other vertices. The discharging rules are such that no vertex sends charge to a vertex at distance greater than $d$, for a fixed integer $d$. If at the end every vertex $v$ has final charge $\mathrm{chg}(v)$ with $q\le\mathrm{chg}(v)\le q'$, where $q$ and $q'$ are rational numbers, then $q\le\dens(C)\le q'$.
\end{theorem}

In our proofs, codewords may start with a charge larger than $1$, which only multiplies all charges by a constant, and charge only moves between vertices at distance at most $2$, so the final charges are bounded.
The general lower bound $2/(\Delta+2)$ for twin-free infinite graphs with the slow growth property and maximum degree $\Delta$, mentioned in the introduction, is proved in~\cite[Theorem~4.1]{SSW24} using Theorem~\ref{thm:discharging} with the most basic discharging rule, often called the \emph{share argument} in the literature: every vertex $v$ receives $1/|I(v)|$ from each codeword of $I(v)$. A codeword $c$ then gives $1$ to at most one vertex of $N[c]$, namely the vertex $v$ with $I(v)=\{c\}$ if it exists, and at most $1/2$ to each other vertex of $N[c]$, hence at most $1+\Delta/2$ in total.

\section{The trihexagonal grid}\label{sec:kagome}

The trihexagonal tiling, whose vertices all have configuration $3.6.3.6$, is the Archimedean grid also known as the \emph{kagome grid}. It is $4$-regular, so the general lower bound of~\cite{SSW24} gives $2/(4+2)=1/3$. We show that this bound is attained, which proves item~(a) of Theorem~\ref{thm:main}.

Note that the trihexagonal grid is an induced subgraph of the triangular grid. Its edges lie on three families of parallel straight lines, and every vertex lies on exactly two of these lines, which belong to different families. In Figure~\ref{fig:kagome}, these families are formed by the horizontal lines and by the lines making angles of $60^\circ$ and $-60^\circ$ with them. We call the lines of the first family \emph{horizontal}, and the vertices lying on no horizontal line \emph{middle vertices}, since each of them lies between two consecutive horizontal lines. Two adjacent vertices in a horizontal line are called \emph{consecutive}. Every vertex of a horizontal line has two (consecutive) neighbors on this line and two middle neighbors, one above and one below the line. Every middle vertex is adjacent to two consecutive vertices of the horizontal line just below it and to two consecutive vertices of the horizontal line just above it. Since every middle vertex has four neighbors on horizontal lines and every vertex of a horizontal line has two middle neighbors, $2/3$ of the vertices lie on horizontal lines.

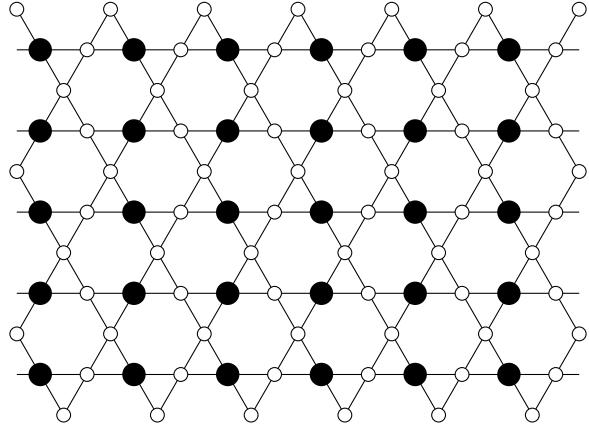
\begin{figure}[t]
\centering
\begin{tikzpicture}
\begin{scope}
\draw (0.000,0.537) -- (7.440,0.537);
\draw (0.000,1.611) -- (7.440,1.611);
\draw (0.000,2.685) -- (7.440,2.685);
\draw (0.000,3.759) -- (7.440,3.759);
\draw (0.000,4.832) -- (7.440,4.832);
\draw (0.620,0.000) -- (0.310,0.537);
\draw (0.620,0.000) -- (0.930,0.537);
\draw (1.860,0.000) -- (1.550,0.537);
\draw (1.860,0.000) -- (2.170,0.537);
\draw (3.100,0.000) -- (2.790,0.537);
\draw (3.100,0.000) -- (3.410,0.537);
\draw (4.340,0.000) -- (4.030,0.537);
\draw (4.340,0.000) -- (4.650,0.537);
\draw (5.580,0.000) -- (5.270,0.537);
\draw (5.580,0.000) -- (5.890,0.537);
\draw (6.820,0.000) -- (6.510,0.537);
\draw (6.820,0.000) -- (7.130,0.537);
\draw (0.000,1.074) -- (0.310,1.611);
\draw (0.000,1.074) -- (0.310,0.537);
\draw (1.240,1.074) -- (0.930,1.611);
\draw (1.240,1.074) -- (1.550,1.611);
\draw (1.240,1.074) -- (0.930,0.537);
\draw (1.240,1.074) -- (1.550,0.537);
\draw (2.480,1.074) -- (2.170,1.611);
\draw (2.480,1.074) -- (2.790,1.611);
\draw (2.480,1.074) -- (2.170,0.537);
\draw (2.480,1.074) -- (2.790,0.537);
\draw (3.720,1.074) -- (3.410,1.611);
\draw (3.720,1.074) -- (4.030,1.611);
\draw (3.720,1.074) -- (3.410,0.537);
\draw (3.720,1.074) -- (4.030,0.537);
\draw (4.960,1.074) -- (4.650,1.611);
\draw (4.960,1.074) -- (5.270,1.611);
\draw (4.960,1.074) -- (4.650,0.537);
\draw (4.960,1.074) -- (5.270,0.537);
\draw (6.200,1.074) -- (5.890,1.611);
\draw (6.200,1.074) -- (6.510,1.611);
\draw (6.200,1.074) -- (5.890,0.537);
\draw (6.200,1.074) -- (6.510,0.537);
\draw (7.440,1.074) -- (7.130,1.611);
\draw (7.440,1.074) -- (7.130,0.537);
\draw (0.620,2.148) -- (0.310,2.685);
\draw (0.620,2.148) -- (0.930,2.685);
\draw (0.620,2.148) -- (0.310,1.611);
\draw (0.620,2.148) -- (0.930,1.611);
\draw (1.860,2.148) -- (1.550,2.685);
\draw (1.860,2.148) -- (2.170,2.685);
\draw (1.860,2.148) -- (1.550,1.611);
\draw (1.860,2.148) -- (2.170,1.611);
\draw (3.100,2.148) -- (2.790,2.685);
\draw (3.100,2.148) -- (3.410,2.685);
\draw (3.100,2.148) -- (2.790,1.611);
\draw (3.100,2.148) -- (3.410,1.611);
\draw (4.340,2.148) -- (4.030,2.685);
\draw (4.340,2.148) -- (4.650,2.685);
\draw (4.340,2.148) -- (4.030,1.611);
\draw (4.340,2.148) -- (4.650,1.611);
\draw (5.580,2.148) -- (5.270,2.685);
\draw (5.580,2.148) -- (5.890,2.685);
\draw (5.580,2.148) -- (5.270,1.611);
\draw (5.580,2.148) -- (5.890,1.611);
\draw (6.820,2.148) -- (6.510,2.685);
\draw (6.820,2.148) -- (7.130,2.685);
\draw (6.820,2.148) -- (6.510,1.611);
\draw (6.820,2.148) -- (7.130,1.611);
\draw (0.000,3.222) -- (0.310,3.759);
\draw (0.000,3.222) -- (0.310,2.685);
\draw (1.240,3.222) -- (0.930,3.759);
\draw (1.240,3.222) -- (1.550,3.759);
\draw (1.240,3.222) -- (0.930,2.685);
\draw (1.240,3.222) -- (1.550,2.685);
\draw (2.480,3.222) -- (2.170,3.759);
\draw (2.480,3.222) -- (2.790,3.759);
\draw (2.480,3.222) -- (2.170,2.685);
\draw (2.480,3.222) -- (2.790,2.685);
\draw (3.720,3.222) -- (3.410,3.759);
\draw (3.720,3.222) -- (4.030,3.759);
\draw (3.720,3.222) -- (3.410,2.685);
\draw (3.720,3.222) -- (4.030,2.685);
\draw (4.960,3.222) -- (4.650,3.759);
\draw (4.960,3.222) -- (5.270,3.759);
\draw (4.960,3.222) -- (4.650,2.685);
\draw (4.960,3.222) -- (5.270,2.685);
\draw (6.200,3.222) -- (5.890,3.759);
\draw (6.200,3.222) -- (6.510,3.759);
\draw (6.200,3.222) -- (5.890,2.685);
\draw (6.200,3.222) -- (6.510,2.685);
\draw (7.440,3.222) -- (7.130,3.759);
\draw (7.440,3.222) -- (7.130,2.685);
\draw (0.620,4.295) -- (0.310,4.832);
\draw (0.620,4.295) -- (0.930,4.832);
\draw (0.620,4.295) -- (0.310,3.759);
\draw (0.620,4.295) -- (0.930,3.759);
\draw (1.860,4.295) -- (1.550,4.832);
\draw (1.860,4.295) -- (2.170,4.832);
\draw (1.860,4.295) -- (1.550,3.759);
\draw (1.860,4.295) -- (2.170,3.759);
\draw (3.100,4.295) -- (2.790,4.832);
\draw (3.100,4.295) -- (3.410,4.832);
\draw (3.100,4.295) -- (2.790,3.759);
\draw (3.100,4.295) -- (3.410,3.759);
\draw (4.340,4.295) -- (4.030,4.832);
\draw (4.340,4.295) -- (4.650,4.832);
\draw (4.340,4.295) -- (4.030,3.759);
\draw (4.340,4.295) -- (4.650,3.759);
\draw (5.580,4.295) -- (5.270,4.832);
\draw (5.580,4.295) -- (5.890,4.832);
\draw (5.580,4.295) -- (5.270,3.759);
\draw (5.580,4.295) -- (5.890,3.759);
\draw (6.820,4.295) -- (6.510,4.832);
\draw (6.820,4.295) -- (7.130,4.832);
\draw (6.820,4.295) -- (6.510,3.759);
\draw (6.820,4.295) -- (7.130,3.759);
\draw (0.000,5.369) -- (0.310,4.832);
\draw (1.240,5.369) -- (0.930,4.832);
\draw (1.240,5.369) -- (1.550,4.832);
\draw (2.480,5.369) -- (2.170,4.832);
\draw (2.480,5.369) -- (2.790,4.832);
\draw (3.720,5.369) -- (3.410,4.832);
\draw (3.720,5.369) -- (4.030,4.832);
\draw (4.960,5.369) -- (4.650,4.832);
\draw (4.960,5.369) -- (5.270,4.832);
\draw (6.200,5.369) -- (5.890,4.832);
\draw (6.200,5.369) -- (6.510,4.832);
\draw (7.440,5.369) -- (7.130,4.832);
\fill[white] (0.000,1.074) circle (2.6pt); \draw (0.000,1.074) circle (2.6pt);
\fill[white] (0.000,3.222) circle (2.6pt); \draw (0.000,3.222) circle (2.6pt);
\fill[white] (0.000,5.369) circle (2.6pt); \draw (0.000,5.369) circle (2.6pt);
\fill (0.310,0.537) circle (4.4pt);
\fill (0.310,1.611) circle (4.4pt);
\fill (0.310,2.685) circle (4.4pt);
\fill (0.310,3.759) circle (4.4pt);
\fill (0.310,4.832) circle (4.4pt);
\fill[white] (0.620,0.000) circle (2.6pt); \draw (0.620,0.000) circle (2.6pt);
\fill[white] (0.620,2.148) circle (2.6pt); \draw (0.620,2.148) circle (2.6pt);
\fill[white] (0.620,4.295) circle (2.6pt); \draw (0.620,4.295) circle (2.6pt);
\fill[white] (0.930,0.537) circle (2.6pt); \draw (0.930,0.537) circle (2.6pt);
\fill[white] (0.930,1.611) circle (2.6pt); \draw (0.930,1.611) circle (2.6pt);
\fill[white] (0.930,2.685) circle (2.6pt); \draw (0.930,2.685) circle (2.6pt);
\fill[white] (0.930,3.759) circle (2.6pt); \draw (0.930,3.759) circle (2.6pt);
\fill[white] (0.930,4.832) circle (2.6pt); \draw (0.930,4.832) circle (2.6pt);
\fill[white] (1.240,1.074) circle (2.6pt); \draw (1.240,1.074) circle (2.6pt);
\fill[white] (1.240,3.222) circle (2.6pt); \draw (1.240,3.222) circle (2.6pt);
\fill[white] (1.240,5.369) circle (2.6pt); \draw (1.240,5.369) circle (2.6pt);
\fill (1.550,0.537) circle (4.4pt);
\fill (1.550,1.611) circle (4.4pt);
\fill (1.550,2.685) circle (4.4pt);
\fill (1.550,3.759) circle (4.4pt);
\fill (1.550,4.832) circle (4.4pt);
\fill[white] (1.860,0.000) circle (2.6pt); \draw (1.860,0.000) circle (2.6pt);
\fill[white] (1.860,2.148) circle (2.6pt); \draw (1.860,2.148) circle (2.6pt);
\fill[white] (1.860,4.295) circle (2.6pt); \draw (1.860,4.295) circle (2.6pt);
\fill[white] (2.170,0.537) circle (2.6pt); \draw (2.170,0.537) circle (2.6pt);
\fill[white] (2.170,1.611) circle (2.6pt); \draw (2.170,1.611) circle (2.6pt);
\fill[white] (2.170,2.685) circle (2.6pt); \draw (2.170,2.685) circle (2.6pt);
\fill[white] (2.170,3.759) circle (2.6pt); \draw (2.170,3.759) circle (2.6pt);
\fill[white] (2.170,4.832) circle (2.6pt); \draw (2.170,4.832) circle (2.6pt);
\fill[white] (2.480,1.074) circle (2.6pt); \draw (2.480,1.074) circle (2.6pt);
\fill[white] (2.480,3.222) circle (2.6pt); \draw (2.480,3.222) circle (2.6pt);
\fill[white] (2.480,5.369) circle (2.6pt); \draw (2.480,5.369) circle (2.6pt);
\fill (2.790,0.537) circle (4.4pt);
\fill (2.790,1.611) circle (4.4pt);
\fill (2.790,2.685) circle (4.4pt);
\fill (2.790,3.759) circle (4.4pt);
\fill (2.790,4.832) circle (4.4pt);
\fill[white] (3.100,0.000) circle (2.6pt); \draw (3.100,0.000) circle (2.6pt);
\fill[white] (3.100,2.148) circle (2.6pt); \draw (3.100,2.148) circle (2.6pt);
\fill[white] (3.100,4.295) circle (2.6pt); \draw (3.100,4.295) circle (2.6pt);
\fill[white] (3.410,0.537) circle (2.6pt); \draw (3.410,0.537) circle (2.6pt);
\fill[white] (3.410,1.611) circle (2.6pt); \draw (3.410,1.611) circle (2.6pt);
\fill[white] (3.410,2.685) circle (2.6pt); \draw (3.410,2.685) circle (2.6pt);
\fill[white] (3.410,3.759) circle (2.6pt); \draw (3.410,3.759) circle (2.6pt);
\fill[white] (3.410,4.832) circle (2.6pt); \draw (3.410,4.832) circle (2.6pt);
\fill[white] (3.720,1.074) circle (2.6pt); \draw (3.720,1.074) circle (2.6pt);
\fill[white] (3.720,3.222) circle (2.6pt); \draw (3.720,3.222) circle (2.6pt);
\fill[white] (3.720,5.369) circle (2.6pt); \draw (3.720,5.369) circle (2.6pt);
\fill (4.030,0.537) circle (4.4pt);
\fill (4.030,1.611) circle (4.4pt);
\fill (4.030,2.685) circle (4.4pt);
\fill (4.030,3.759) circle (4.4pt);
\fill (4.030,4.832) circle (4.4pt);
\fill[white] (4.340,0.000) circle (2.6pt); \draw (4.340,0.000) circle (2.6pt);
\fill[white] (4.340,2.148) circle (2.6pt); \draw (4.340,2.148) circle (2.6pt);
\fill[white] (4.340,4.295) circle (2.6pt); \draw (4.340,4.295) circle (2.6pt);
\fill[white] (4.650,0.537) circle (2.6pt); \draw (4.650,0.537) circle (2.6pt);
\fill[white] (4.650,1.611) circle (2.6pt); \draw (4.650,1.611) circle (2.6pt);
\fill[white] (4.650,2.685) circle (2.6pt); \draw (4.650,2.685) circle (2.6pt);
\fill[white] (4.650,3.759) circle (2.6pt); \draw (4.650,3.759) circle (2.6pt);
\fill[white] (4.650,4.832) circle (2.6pt); \draw (4.650,4.832) circle (2.6pt);
\fill[white] (4.960,1.074) circle (2.6pt); \draw (4.960,1.074) circle (2.6pt);
\fill[white] (4.960,3.222) circle (2.6pt); \draw (4.960,3.222) circle (2.6pt);
\fill[white] (4.960,5.369) circle (2.6pt); \draw (4.960,5.369) circle (2.6pt);
\fill (5.270,0.537) circle (4.4pt);
\fill (5.270,1.611) circle (4.4pt);
\fill (5.270,2.685) circle (4.4pt);
\fill (5.270,3.759) circle (4.4pt);
\fill (5.270,4.832) circle (4.4pt);
\fill[white] (5.580,0.000) circle (2.6pt); \draw (5.580,0.000) circle (2.6pt);
\fill[white] (5.580,2.148) circle (2.6pt); \draw (5.580,2.148) circle (2.6pt);
\fill[white] (5.580,4.295) circle (2.6pt); \draw (5.580,4.295) circle (2.6pt);
\fill[white] (5.890,0.537) circle (2.6pt); \draw (5.890,0.537) circle (2.6pt);
\fill[white] (5.890,1.611) circle (2.6pt); \draw (5.890,1.611) circle (2.6pt);
\fill[white] (5.890,2.685) circle (2.6pt); \draw (5.890,2.685) circle (2.6pt);
\fill[white] (5.890,3.759) circle (2.6pt); \draw (5.890,3.759) circle (2.6pt);
\fill[white] (5.890,4.832) circle (2.6pt); \draw (5.890,4.832) circle (2.6pt);
\fill[white] (6.200,1.074) circle (2.6pt); \draw (6.200,1.074) circle (2.6pt);
\fill[white] (6.200,3.222) circle (2.6pt); \draw (6.200,3.222) circle (2.6pt);
\fill[white] (6.200,5.369) circle (2.6pt); \draw (6.200,5.369) circle (2.6pt);
\fill (6.510,0.537) circle (4.4pt);
\fill (6.510,1.611) circle (4.4pt);
\fill (6.510,2.685) circle (4.4pt);
\fill (6.510,3.759) circle (4.4pt);
\fill (6.510,4.832) circle (4.4pt);
\fill[white] (6.820,0.000) circle (2.6pt); \draw (6.820,0.000) circle (2.6pt);
\fill[white] (6.820,2.148) circle (2.6pt); \draw (6.820,2.148) circle (2.6pt);
\fill[white] (6.820,4.295) circle (2.6pt); \draw (6.820,4.295) circle (2.6pt);
\fill[white] (7.130,0.537) circle (2.6pt); \draw (7.130,0.537) circle (2.6pt);
\fill[white] (7.130,1.611) circle (2.6pt); \draw (7.130,1.611) circle (2.6pt);
\fill[white] (7.130,2.685) circle (2.6pt); \draw (7.130,2.685) circle (2.6pt);
\fill[white] (7.130,3.759) circle (2.6pt); \draw (7.130,3.759) circle (2.6pt);
\fill[white] (7.130,4.832) circle (2.6pt); \draw (7.130,4.832) circle (2.6pt);
\fill[white] (7.440,1.074) circle (2.6pt); \draw (7.440,1.074) circle (2.6pt);
\fill[white] (7.440,3.222) circle (2.6pt); \draw (7.440,3.222) circle (2.6pt);
\fill[white] (7.440,5.369) circle (2.6pt); \draw (7.440,5.369) circle (2.6pt);
\end{scope}
\end{tikzpicture}
\caption{A patch of the kagome grid with five horizontal lines. The black vertices form the idcode of density $1/3$ of the proof of Theorem~\ref{thm:kagome}: no middle vertex, and codewords and non-codewords alternate along each horizontal line. The codewords form vertical columns.}
\label{fig:kagome}
\end{figure}

\begin{theorem}\label{thm:kagome}
The minimum density of an idcode of the trihexagonal grid $3.6.3.6$ is $1/3$.
\end{theorem}

\begin{proof}
The trihexagonal grid is $4$-regular, twin-free, and has the slow growth property, so every idcode has density at least $2/(4+2)=1/3$ by~\cite[Theorem~4.1]{SSW24}. Conversely, let $C$ be a set of vertices containing no middle vertex and such that codewords and non-codewords alternate along each horizontal line, as in Figure~\ref{fig:kagome}. If $v\in C$, then $I(v)=\{v\}$. If $v\notin C$ lies on a horizontal line, then $I(v)$ consists of the two neighbors of $v$ on this line, and $v$ is their only common neighbor, since a middle vertex is adjacent only to consecutive vertices of a line. If $v$ is a middle vertex, then $I(v)$ consists of one codeword on the horizontal line just below $v$ and one on the line just above $v$, and again $v$ is their only common neighbor. Hence the sets $I(v)$ are nonempty and pairwise distinct, and $C$ is an idcode. Its density is $1/3$, since $C$ contains half of the vertices of the horizontal lines, which are $2/3$ of all vertices.
\end{proof}

The general bound is tight here, in contrast with the other grids treated in this paper. In the codes of the proof of Theorem~\ref{thm:kagome}, every codeword $c$ has $I(c)=\{c\}$ and every vertex $v\notin C$ has $|I(v)|=2$, so every codeword gives exactly $1+4/2=3$ in the share argument. Since the choice of the parity class on each horizontal line is free, there are uncountably many idcodes of density $1/3$. Note also that the code of Figure~\ref{fig:kagome} is not an idcode of the triangular grid (which has an extra vertex at the center of every hexagon of Figure~\ref{fig:kagome}): every such center has the same set $I$ as one of its middle neighbors.

\section{The truncated hexagonal grid}\label{sec:trhex}

The truncated hexagonal tiling has vertex configuration $3.12.12$ (see Figure~\ref{fig:trhex}). It is the truncation $\operatorname{tr}(6^3)$: every vertex of the hexagonal grid is replaced by a triangle. The grid is cubic, so the general lower bound is only $2/5$. In this section we prove that the minimum density is $2/3$, which is item~(b) of Theorem~\ref{thm:main}.

We first prove a general lower bound for every cubic graph in which every vertex lies on exactly one triangle, such as the truncation of any cubic graph. In such a graph, every vertex $v$ lies on exactly one triangle $T(v)$ and has exactly one neighbor outside it, its \emph{bridge neighbor} $\beta(v)$, and $\beta(\beta(v))=v$. The bridge neighbors of the three vertices of a triangle are distinct, since a common bridge neighbor would form a second triangle. Such a graph is twin-free, since $N[u]=N[v]$ for distinct vertices $u,v$ would put $u$ on two triangles.

\begin{figure}[t]
\centering
\begin{tikzpicture}[scale=0.62]
\draw (0.500,3.232)--(0.000,4.098);
\draw (0.500,3.232)--(1.000,4.098);
\draw (0.500,3.232)--(0.500,2.232);
\draw (0.000,4.098)--(-0.693,4.498);
\draw (0.000,4.098)--(1.000,4.098);
\draw (2.366,6.464)--(1.866,7.330);
\draw (2.366,6.464)--(2.866,7.330);
\draw (2.366,6.464)--(2.366,5.464);
\draw (1.866,7.330)--(1.000,7.830);
\draw (1.866,7.330)--(2.866,7.330);
\draw (1.000,7.830)--(0.200,7.830);
\draw (1.000,7.830)--(0.600,8.523);
\draw (2.366,0.000)--(1.866,0.866);
\draw (2.366,0.000)--(2.866,0.866);
\draw (2.366,0.000)--(2.366,-0.450);
\draw (1.866,0.866)--(1.000,1.366);
\draw (1.866,0.866)--(2.866,0.866);
\draw (1.000,1.366)--(0.200,1.366);
\draw (1.000,1.366)--(0.500,2.232);
\draw (4.232,3.232)--(3.732,4.098);
\draw (4.232,3.232)--(4.732,4.098);
\draw (4.232,3.232)--(4.232,2.232);
\draw (3.732,4.098)--(2.866,4.598);
\draw (3.732,4.098)--(4.732,4.098);
\draw (2.866,4.598)--(1.866,4.598);
\draw (2.866,4.598)--(2.366,5.464);
\draw (1.866,4.598)--(1.000,4.098);
\draw (1.866,4.598)--(2.366,5.464);
\draw (0.500,2.232)--(0.100,1.539);
\draw (6.098,6.464)--(5.598,7.330);
\draw (6.098,6.464)--(6.598,7.330);
\draw (6.098,6.464)--(6.098,5.464);
\draw (5.598,7.330)--(4.732,7.830);
\draw (5.598,7.330)--(6.598,7.330);
\draw (4.732,7.830)--(3.732,7.830);
\draw (4.732,7.830)--(4.332,8.523);
\draw (3.732,7.830)--(2.866,7.330);
\draw (3.732,7.830)--(4.132,8.523);
\draw (6.098,0.000)--(5.598,0.866);
\draw (6.098,0.000)--(6.598,0.866);
\draw (6.098,0.000)--(6.098,-0.450);
\draw (5.598,0.866)--(4.732,1.366);
\draw (5.598,0.866)--(6.598,0.866);
\draw (4.732,1.366)--(3.732,1.366);
\draw (4.732,1.366)--(4.232,2.232);
\draw (3.732,1.366)--(2.866,0.866);
\draw (3.732,1.366)--(4.232,2.232);
\draw (7.964,3.232)--(7.464,4.098);
\draw (7.964,3.232)--(8.464,4.098);
\draw (7.964,3.232)--(7.964,2.232);
\draw (7.464,4.098)--(6.598,4.598);
\draw (7.464,4.098)--(8.464,4.098);
\draw (6.598,4.598)--(5.598,4.598);
\draw (6.598,4.598)--(6.098,5.464);
\draw (5.598,4.598)--(4.732,4.098);
\draw (5.598,4.598)--(6.098,5.464);
\draw (9.830,6.464)--(9.330,7.330);
\draw (9.830,6.464)--(10.330,7.330);
\draw (9.830,6.464)--(9.830,5.464);
\draw (9.330,7.330)--(8.464,7.830);
\draw (9.330,7.330)--(10.330,7.330);
\draw (8.464,7.830)--(7.464,7.830);
\draw (8.464,7.830)--(8.064,8.523);
\draw (7.464,7.830)--(6.598,7.330);
\draw (7.464,7.830)--(7.864,8.523);
\draw (9.830,0.000)--(9.330,0.866);
\draw (9.830,0.000)--(10.330,0.866);
\draw (9.830,0.000)--(9.830,-0.450);
\draw (9.330,0.866)--(8.464,1.366);
\draw (9.330,0.866)--(10.330,0.866);
\draw (8.464,1.366)--(7.464,1.366);
\draw (8.464,1.366)--(7.964,2.232);
\draw (7.464,1.366)--(6.598,0.866);
\draw (7.464,1.366)--(7.964,2.232);
\draw (11.696,3.232)--(11.196,4.098);
\draw (11.696,3.232)--(12.196,4.098);
\draw (11.696,3.232)--(11.696,2.232);
\draw (11.196,4.098)--(10.330,4.598);
\draw (11.196,4.098)--(12.196,4.098);
\draw (10.330,4.598)--(9.330,4.598);
\draw (10.330,4.598)--(9.830,5.464);
\draw (9.330,4.598)--(8.464,4.098);
\draw (9.330,4.598)--(9.830,5.464);
\draw (13.562,6.464)--(13.062,7.330);
\draw (13.562,6.464)--(13.962,7.157);
\draw (13.562,6.464)--(13.562,5.464);
\draw (13.062,7.330)--(12.196,7.830);
\draw (13.062,7.330)--(13.862,7.330);
\draw (12.196,7.830)--(11.196,7.830);
\draw (12.196,7.830)--(11.796,8.523);
\draw (11.196,7.830)--(10.330,7.330);
\draw (11.196,7.830)--(11.596,8.523);
\draw (13.562,0.000)--(13.062,0.866);
\draw (13.562,0.000)--(13.962,0.693);
\draw (13.562,0.000)--(13.562,-0.450);
\draw (13.062,0.866)--(12.196,1.366);
\draw (13.062,0.866)--(13.862,0.866);
\draw (12.196,1.366)--(11.196,1.366);
\draw (12.196,1.366)--(11.696,2.232);
\draw (11.196,1.366)--(10.330,0.866);
\draw (11.196,1.366)--(11.696,2.232);
\draw (13.062,4.598)--(13.862,4.598);
\draw (13.062,4.598)--(12.196,4.098);
\draw (13.062,4.598)--(13.562,5.464);
\draw (13.562,5.464)--(13.962,4.771);
\fill (0.500,3.232) circle (4.35pt);
\fill[white] (0.000,4.098) circle (2.58pt); \draw (0.000,4.098) circle (2.58pt);
\fill (2.366,6.464) circle (4.35pt);
\fill[white] (1.866,7.330) circle (2.58pt); \draw (1.866,7.330) circle (2.58pt);
\fill (1.000,7.830) circle (4.35pt);
\fill (2.366,0.000) circle (4.35pt);
\fill[white] (1.866,0.866) circle (2.58pt); \draw (1.866,0.866) circle (2.58pt);
\fill (1.000,1.366) circle (4.35pt);
\fill (4.232,3.232) circle (4.35pt);
\fill[white] (3.732,4.098) circle (2.58pt); \draw (3.732,4.098) circle (2.58pt);
\fill (2.866,4.598) circle (4.35pt);
\fill[white] (1.866,4.598) circle (2.58pt); \draw (1.866,4.598) circle (2.58pt);
\fill (1.000,4.098) circle (4.35pt);
\fill (0.500,2.232) circle (4.35pt);
\fill (6.098,6.464) circle (4.35pt);
\fill[white] (5.598,7.330) circle (2.58pt); \draw (5.598,7.330) circle (2.58pt);
\fill (4.732,7.830) circle (4.35pt);
\fill[white] (3.732,7.830) circle (2.58pt); \draw (3.732,7.830) circle (2.58pt);
\fill (2.866,7.330) circle (4.35pt);
\fill (2.366,5.464) circle (4.35pt);
\fill (6.098,0.000) circle (4.35pt);
\fill[white] (5.598,0.866) circle (2.58pt); \draw (5.598,0.866) circle (2.58pt);
\fill (4.732,1.366) circle (4.35pt);
\fill[white] (3.732,1.366) circle (2.58pt); \draw (3.732,1.366) circle (2.58pt);
\fill (2.866,0.866) circle (4.35pt);
\fill (7.964,3.232) circle (4.35pt);
\fill[white] (7.464,4.098) circle (2.58pt); \draw (7.464,4.098) circle (2.58pt);
\fill (6.598,4.598) circle (4.35pt);
\fill[white] (5.598,4.598) circle (2.58pt); \draw (5.598,4.598) circle (2.58pt);
\fill (4.732,4.098) circle (4.35pt);
\fill (4.232,2.232) circle (4.35pt);
\fill (9.830,6.464) circle (4.35pt);
\fill[white] (9.330,7.330) circle (2.58pt); \draw (9.330,7.330) circle (2.58pt);
\fill (8.464,7.830) circle (4.35pt);
\fill[white] (7.464,7.830) circle (2.58pt); \draw (7.464,7.830) circle (2.58pt);
\fill (6.598,7.330) circle (4.35pt);
\fill (6.098,5.464) circle (4.35pt);
\fill (9.830,0.000) circle (4.35pt);
\fill[white] (9.330,0.866) circle (2.58pt); \draw (9.330,0.866) circle (2.58pt);
\fill (8.464,1.366) circle (4.35pt);
\fill[white] (7.464,1.366) circle (2.58pt); \draw (7.464,1.366) circle (2.58pt);
\fill (6.598,0.866) circle (4.35pt);
\fill (11.696,3.232) circle (4.35pt);
\fill[white] (11.196,4.098) circle (2.58pt); \draw (11.196,4.098) circle (2.58pt);
\fill (10.330,4.598) circle (4.35pt);
\fill[white] (9.330,4.598) circle (2.58pt); \draw (9.330,4.598) circle (2.58pt);
\fill (8.464,4.098) circle (4.35pt);
\fill (7.964,2.232) circle (4.35pt);
\fill (13.562,6.464) circle (4.35pt);
\fill[white] (13.062,7.330) circle (2.58pt); \draw (13.062,7.330) circle (2.58pt);
\fill (12.196,7.830) circle (4.35pt);
\fill[white] (11.196,7.830) circle (2.58pt); \draw (11.196,7.830) circle (2.58pt);
\fill (10.330,7.330) circle (4.35pt);
\fill (9.830,5.464) circle (4.35pt);
\fill (13.562,0.000) circle (4.35pt);
\fill[white] (13.062,0.866) circle (2.58pt); \draw (13.062,0.866) circle (2.58pt);
\fill (12.196,1.366) circle (4.35pt);
\fill[white] (11.196,1.366) circle (2.58pt); \draw (11.196,1.366) circle (2.58pt);
\fill (10.330,0.866) circle (4.35pt);
\fill[white] (13.062,4.598) circle (2.58pt); \draw (13.062,4.598) circle (2.58pt);
\fill (12.196,4.098) circle (4.35pt);
\fill (11.696,2.232) circle (4.35pt);
\fill (13.562,5.464) circle (4.35pt);
\end{tikzpicture}
\caption{The idcode of density $2/3$ of the truncated hexagonal grid of the proof of Theorem~\ref{thm:trhex} (black vertices). In every triangle, the white vertex is the exit: it is joined by a non-vertical edge to the triangle on its left.}
\label{fig:trhex}
\end{figure}
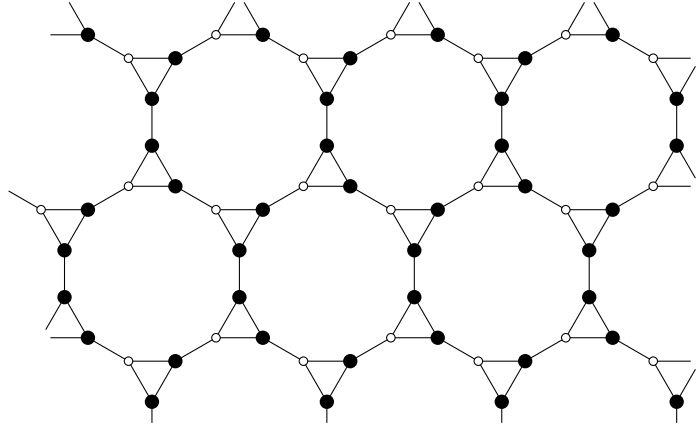

\begin{theorem}\label{thm:3cycle}
Every idcode of a cubic graph with the slow growth property in which every vertex lies on exactly one triangle has density at least $2/3$.
\end{theorem}

\begin{proof}
Let $C$ be an idcode, and let $T=\{u,v,w\}$ be a triangle. Since $N[u]\triangle N[v]$ is the set $\{\beta(u),\beta(v)\}$, the code $C$ contains $\beta(u)$ or $\beta(v)$, and similarly for the pairs $u,w$ and $v,w$. Hence at least two of $\beta(u),\beta(v),\beta(w)$ are codewords. Every codeword starts with charge $1$ and all other vertices with charge $0$, and every codeword $c$ sends $1/3$ to each vertex of the triangle $T(\beta(c))$, which is at distance at most $2$ from $c$. Every vertex $v$ then receives $1/3$ for every codeword among the bridge neighbors of the vertices of $T(v)$, hence at least $2/3$, and Theorem~\ref{thm:discharging} gives $\dens(C)\ge 2/3$.
\end{proof}

For example, Theorem~\ref{thm:3cycle} applies to the truncation $\operatorname{tr}(4.8.8)$, a non-Archimedean cubic grid in which every vertex lies on exactly one triangle and has configuration $3.8.16$ or $3.16.16$.

We now turn to the truncated hexagonal grid $3.12.12$. We draw the grid so that one third of the edges $v\beta(v)$ are vertical (Figure~\ref{fig:trhex}). Then every triangle has one vertex $z$ with $z\beta(z)$ vertical, one vertex $x$ with $\beta(x)$ to its left and one vertex $y$ with $\beta(y)$ to its right. We call $x$ the \emph{exit} of the triangle.

\begin{theorem}\label{thm:trhex}
The minimum density of an idcode of the truncated hexagonal grid $3.12.12$ is $2/3$.
\end{theorem}

\begin{proof}
The truncated hexagonal grid is cubic, and every vertex lies on exactly one triangle, so Theorem~\ref{thm:3cycle} gives the lower bound.
Conversely, let $C$ be the set of all vertices that are not exits, as in Figure~\ref{fig:trhex}. Every triangle contains exactly one exit, so $C$ has density $2/3$. Let $T=\{x,y,z\}$ be a triangle, named as above. Then $x$ is the exit of $T$ and $\beta(y)$ is the exit of the triangle to the right of $T$, while $\beta(x)$ and $\beta(z)$ are not exits. Hence
\[
I(x)=\{y,z,\beta(x)\},\qquad I(y)=\{y,z\},\qquad I(z)=\{y,z,\beta(z)\}.
\]
These sets are nonempty and pairwise distinct, since $\beta(x)\neq \beta(z)$. Every set $I(v)$ contains two vertices of $T(v)$, and the only vertices adjacent to two vertices of a triangle $T$ are the vertices of $T$, since the bridge neighbors of the vertices of $T$ are distinct. Hence $I(v)$ determines $T(v)$, and therefore $v$, and $C$ is an idcode.
\end{proof}

The bound $2/3$ is much larger than the general bound $2/5$: the triangles force many codewords, and the argument above only uses the separation of adjacent vertices.

\section{The truncated square grid}\label{sec:tsq}

The truncated square tiling has vertex configuration $4.8.8$ (see Figure~\ref{fig:code}). It is the truncation $\operatorname{tr}(4^4)$: every vertex of the square grid is replaced by a $4$-cycle. The grid is cubic, so the general lower bound is only $2/5$. In this section we prove that the minimum density is $3/7$, which is item~(c) of Theorem~\ref{thm:main}.

We first prove a general lower bound for every twin-free cubic graph in which every vertex lies on exactly one $4$-cycle, such as the truncation of any $4$-regular graph; in Section~\ref{sec:remaining} we also apply it to the truncated trihexagonal grid $4.6.12=\operatorname{tr}(3.6.3.6)$. In such a graph, the \emph{square} of a vertex $v$ is the $4$-cycle containing it. A square has no chord, since a chord $v_1v_3$ of a square $v_1,v_2,v_3,v_4$ would give $N[v_1]=N[v_3]$. Thus the two neighbors of $v$ on its square are its \emph{square neighbors}, and the third neighbor of $v$ is its \emph{bridge neighbor} $\beta(v)$. The edge $v\beta(v)$ lies on no $4$-cycle, and $\beta(\beta(v))=v$. If $v_1,v_2,v_3,v_4$ is a square in cyclic order, then the square neighbors of $v_1$ and $v_3$ are $v_2$ and $v_4$, and vice versa.

\begin{theorem}\label{thm:4cycle}
Every idcode of a cubic graph with the slow growth property in which every vertex lies on exactly one $4$-cycle has density at least $3/7$.
\end{theorem}

\begin{proof}
Let $G$ be such a graph and let $C$ be an idcode of $G$. Every codeword starts with charge $7$ and all other vertices with charge $0$. For a vertex $v$, let $S(v)$ be the set of square neighbors of $v$ that belong to $C$. In the discharging phase, each codeword $u$ sends charge to its neighbors according to the following rules.
\begin{itemize}
\item[(R1)] If $b=\beta(u)\notin C$, the codeword $u$ sends charge $3$ or $2$ to $b$, depending on whether $S(b)=\emptyset$ or not.
\item[(R2)] Let $s$ be a square neighbor of $u$, and let $\ell=|S(s)|\ge 1$. If $s\notin C$, then $u$ sends $1/\ell$ or $3/\ell$ to $s$, depending on whether $\beta(s)\in C$ or not. If $s\in C$ and $\beta(s)\notin C$, then $u$ sends $1/\ell$ to $s$.
\end{itemize}
In words, a vertex is served first by itself, then by its bridge neighbor, and its square neighbors complete what is missing in equal parts. A codeword $u$ receives nothing if $\beta(u)\in C$ or $S(u)=\emptyset$. Otherwise it receives $1/|S(u)|$ from each vertex of $S(u)$, hence $1$ in total. A vertex $v\notin C$ receives exactly $3$. Indeed, if $\beta(v)\in C$, it receives $3$ from $\beta(v)$ when $S(v)=\emptyset$, and otherwise $2$ from $\beta(v)$ and $1/|S(v)|$ from each vertex of $S(v)$, hence $2+1=3$ in total. If $\beta(v)\notin C$, it receives $3/|S(v)|$ from each vertex of $S(v)$, which is nonempty because $C$ is dominating.

It remains to show that every codeword ends with charge at least $3$. Let $u\in C$, let $u,s_1,w,s_2$ be its square in cyclic order, let $b=\beta(u)$ and $b_i=\beta(s_i)$ for $i=1,2$. Denote by $q_b$, $q_1$, $q_2$ the charges that $u$ sends to $b$, $s_1$, $s_2$. The final charge of $u$ is $7-q_b-q_1-q_2+\delta$, where $\delta=1$ if $b\notin C$ and at least one of $s_1,s_2$ is in $C$, and $\delta=0$ otherwise. So we must show that $q_b+q_1+q_2\le 4+\delta$.

By (R1), $q_b=0$ if $b\in C$. If $b\notin C$, then $q_b=3$ when $I(b)=\{u\}$ and $q_b=2$ otherwise. The square neighbors of $s_i$ are $u$ and $w$. If $s_i\in C$, rule (R2) gives $q_i=0$ when $b_i\in C$ and $q_i\le 1$ in any case. If $s_i\notin C$, rule (R2) gives $q_i=3$, $3/2$, $1$ or $1/2$ according to whether $I(s_i)$ is $\{u\}$, $\{u,w\}$, $\{u,b_i\}$ or $\{u,w,b_i\}$. Thus $q_i\le 3$ always, $q_i\le 3/2$ when $w\in C$, and $q_i\le 1$ when $b_i\in C$.

Suppose first that $b\in C$, so that $\delta=q_b=0$ and we must show $q_1+q_2\le 4$. If some $s_i$ is a codeword, then $q_i\le 1$ and $q_{3-i}\le 3$. If $s_1,s_2\notin C$ and $w\in C$, then $q_1+q_2\le 3$. If $s_1,s_2,w\notin C$, then $I(s_i)=\{u\}\cup(\{b_i\}\cap C)$, and $I(s_1)\neq I(s_2)$ puts one of $b_1,b_2$ in $C$, so $q_1+q_2\le 3+1$.

Suppose now that $b\notin C$. If $s_1,s_2\in C$, then $\delta=1$ and $q_b+q_1+q_2\le 3+1+1=5$. If $s_1,s_2\notin C$, then $\delta=0$ and $I(u)=\{u\}$, so $I(b)\neq I(u)$ gives $q_b=2$. When $w\notin C$ we have $I(s_i)=\{u\}\cup(\{b_i\}\cap C)$, and $I(s_i)\neq I(u)$ forces $b_i\in C$, so $q_1=q_2=1$. When $w\in C$ we have $I(s_i)=\{u,w\}\cup(\{b_i\}\cap C)$, and $I(s_1)\neq I(s_2)$ puts one of $b_1,b_2$ in $C$, so $q_1+q_2\le 3/2+1/2$. In both cases the sum is at most $4$.

It remains the case where exactly one of $s_1,s_2$, say $s_1$, is a codeword. Then $\delta=1$ and $I(u)=\{u,s_1\}$, and we must show that the sum is at most $5$. If $w\in C$, then $q_1\le 1/2$ and $q_2\le 3/2$, and the sum is at most $3+1/2+3/2=5$. If $w\notin C$ and $b_2\in C$, then $q_1,q_2\le 1$ and the sum is at most $5$. Finally, suppose $w,b_2\notin C$. Then $I(s_2)=\{u\}$, so $q_2=3$, and $I(b)\neq I(s_2)$ gives $q_b=2$. Moreover $I(s_1)=\{u,s_1\}\cup(\{b_1\}\cap C)$, and $I(s_1)\neq I(u)$ forces $b_1\in C$, so $q_1=0$ and the sum is exactly $5$.

Hence every vertex ends with charge at least $3$. Every codeword started with charge $7$ and charge only moves between adjacent vertices, so Theorem~\ref{thm:discharging} gives density at least $3/7$.
\end{proof}

For example, Theorem~\ref{thm:4cycle} applies to the truncation $\operatorname{tr}(3.4.6.4)$, a non-Archimedean cubic grid in which every vertex lies on exactly one $4$-cycle and has configuration $4.6.8$ or $4.8.12$. We now turn to the truncated square grid $4.8.8$.

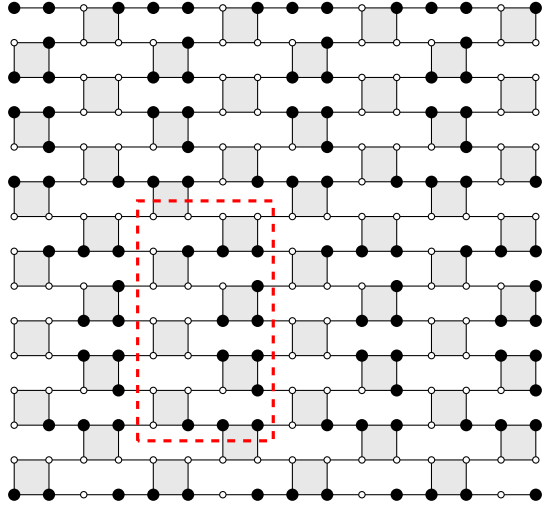
\begin{figure}[t]
\centering
\begin{tikzpicture}[scale=0.46]
\fill[gray!18] (-4,-2) rectangle (-3,-1);
\fill[gray!18] (-4,0) rectangle (-3,1);
\fill[gray!18] (-4,2) rectangle (-3,3);
\fill[gray!18] (-4,4) rectangle (-3,5);
\fill[gray!18] (-4,6) rectangle (-3,7);
\fill[gray!18] (-4,8) rectangle (-3,9);
\fill[gray!18] (-4,10) rectangle (-3,11);
\fill[gray!18] (-2,-1) rectangle (-1,0);
\fill[gray!18] (-2,1) rectangle (-1,2);
\fill[gray!18] (-2,3) rectangle (-1,4);
\fill[gray!18] (-2,5) rectangle (-1,6);
\fill[gray!18] (-2,7) rectangle (-1,8);
\fill[gray!18] (-2,9) rectangle (-1,10);
\fill[gray!18] (-2,11) rectangle (-1,12);
\fill[gray!18] (0,-2) rectangle (1,-1);
\fill[gray!18] (0,0) rectangle (1,1);
\fill[gray!18] (0,2) rectangle (1,3);
\fill[gray!18] (0,4) rectangle (1,5);
\fill[gray!18] (0,6) rectangle (1,7);
\fill[gray!18] (0,8) rectangle (1,9);
\fill[gray!18] (0,10) rectangle (1,11);
\fill[gray!18] (2,-1) rectangle (3,0);
\fill[gray!18] (2,1) rectangle (3,2);
\fill[gray!18] (2,3) rectangle (3,4);
\fill[gray!18] (2,5) rectangle (3,6);
\fill[gray!18] (2,7) rectangle (3,8);
\fill[gray!18] (2,9) rectangle (3,10);
\fill[gray!18] (2,11) rectangle (3,12);
\fill[gray!18] (4,-2) rectangle (5,-1);
\fill[gray!18] (4,0) rectangle (5,1);
\fill[gray!18] (4,2) rectangle (5,3);
\fill[gray!18] (4,4) rectangle (5,5);
\fill[gray!18] (4,6) rectangle (5,7);
\fill[gray!18] (4,8) rectangle (5,9);
\fill[gray!18] (4,10) rectangle (5,11);
\fill[gray!18] (6,-1) rectangle (7,0);
\fill[gray!18] (6,1) rectangle (7,2);
\fill[gray!18] (6,3) rectangle (7,4);
\fill[gray!18] (6,5) rectangle (7,6);
\fill[gray!18] (6,7) rectangle (7,8);
\fill[gray!18] (6,9) rectangle (7,10);
\fill[gray!18] (6,11) rectangle (7,12);
\fill[gray!18] (8,-2) rectangle (9,-1);
\fill[gray!18] (8,0) rectangle (9,1);
\fill[gray!18] (8,2) rectangle (9,3);
\fill[gray!18] (8,4) rectangle (9,5);
\fill[gray!18] (8,6) rectangle (9,7);
\fill[gray!18] (8,8) rectangle (9,9);
\fill[gray!18] (8,10) rectangle (9,11);
\fill[gray!18] (10,-1) rectangle (11,0);
\fill[gray!18] (10,1) rectangle (11,2);
\fill[gray!18] (10,3) rectangle (11,4);
\fill[gray!18] (10,5) rectangle (11,6);
\fill[gray!18] (10,7) rectangle (11,8);
\fill[gray!18] (10,9) rectangle (11,10);
\fill[gray!18] (10,11) rectangle (11,12);
\draw (-4,-2) -- (11,-2);
\draw (-4,-1) -- (11,-1);
\draw (-4,0) -- (11,0);
\draw (-4,1) -- (11,1);
\draw (-4,2) -- (11,2);
\draw (-4,3) -- (11,3);
\draw (-4,4) -- (11,4);
\draw (-4,5) -- (11,5);
\draw (-4,6) -- (11,6);
\draw (-4,7) -- (11,7);
\draw (-4,8) -- (11,8);
\draw (-4,9) -- (11,9);
\draw (-4,10) -- (11,10);
\draw (-4,11) -- (11,11);
\draw (-4,12) -- (11,12);
\draw (-4,-2) -- (-4,-1);
\draw (-4,0) -- (-4,1);
\draw (-4,2) -- (-4,3);
\draw (-4,4) -- (-4,5);
\draw (-4,6) -- (-4,7);
\draw (-4,8) -- (-4,9);
\draw (-4,10) -- (-4,11);
\draw (-3,-2) -- (-3,-1);
\draw (-3,0) -- (-3,1);
\draw (-3,2) -- (-3,3);
\draw (-3,4) -- (-3,5);
\draw (-3,6) -- (-3,7);
\draw (-3,8) -- (-3,9);
\draw (-3,10) -- (-3,11);
\draw (-2,-1) -- (-2,0);
\draw (-2,1) -- (-2,2);
\draw (-2,3) -- (-2,4);
\draw (-2,5) -- (-2,6);
\draw (-2,7) -- (-2,8);
\draw (-2,9) -- (-2,10);
\draw (-2,11) -- (-2,12);
\draw (-1,-1) -- (-1,0);
\draw (-1,1) -- (-1,2);
\draw (-1,3) -- (-1,4);
\draw (-1,5) -- (-1,6);
\draw (-1,7) -- (-1,8);
\draw (-1,9) -- (-1,10);
\draw (-1,11) -- (-1,12);
\draw (0,-2) -- (0,-1);
\draw (0,0) -- (0,1);
\draw (0,2) -- (0,3);
\draw (0,4) -- (0,5);
\draw (0,6) -- (0,7);
\draw (0,8) -- (0,9);
\draw (0,10) -- (0,11);
\draw (1,-2) -- (1,-1);
\draw (1,0) -- (1,1);
\draw (1,2) -- (1,3);
\draw (1,4) -- (1,5);
\draw (1,6) -- (1,7);
\draw (1,8) -- (1,9);
\draw (1,10) -- (1,11);
\draw (2,-1) -- (2,0);
\draw (2,1) -- (2,2);
\draw (2,3) -- (2,4);
\draw (2,5) -- (2,6);
\draw (2,7) -- (2,8);
\draw (2,9) -- (2,10);
\draw (2,11) -- (2,12);
\draw (3,-1) -- (3,0);
\draw (3,1) -- (3,2);
\draw (3,3) -- (3,4);
\draw (3,5) -- (3,6);
\draw (3,7) -- (3,8);
\draw (3,9) -- (3,10);
\draw (3,11) -- (3,12);
\draw (4,-2) -- (4,-1);
\draw (4,0) -- (4,1);
\draw (4,2) -- (4,3);
\draw (4,4) -- (4,5);
\draw (4,6) -- (4,7);
\draw (4,8) -- (4,9);
\draw (4,10) -- (4,11);
\draw (5,-2) -- (5,-1);
\draw (5,0) -- (5,1);
\draw (5,2) -- (5,3);
\draw (5,4) -- (5,5);
\draw (5,6) -- (5,7);
\draw (5,8) -- (5,9);
\draw (5,10) -- (5,11);
\draw (6,-1) -- (6,0);
\draw (6,1) -- (6,2);
\draw (6,3) -- (6,4);
\draw (6,5) -- (6,6);
\draw (6,7) -- (6,8);
\draw (6,9) -- (6,10);
\draw (6,11) -- (6,12);
\draw (7,-1) -- (7,0);
\draw (7,1) -- (7,2);
\draw (7,3) -- (7,4);
\draw (7,5) -- (7,6);
\draw (7,7) -- (7,8);
\draw (7,9) -- (7,10);
\draw (7,11) -- (7,12);
\draw (8,-2) -- (8,-1);
\draw (8,0) -- (8,1);
\draw (8,2) -- (8,3);
\draw (8,4) -- (8,5);
\draw (8,6) -- (8,7);
\draw (8,8) -- (8,9);
\draw (8,10) -- (8,11);
\draw (9,-2) -- (9,-1);
\draw (9,0) -- (9,1);
\draw (9,2) -- (9,3);
\draw (9,4) -- (9,5);
\draw (9,6) -- (9,7);
\draw (9,8) -- (9,9);
\draw (9,10) -- (9,11);
\draw (10,-1) -- (10,0);
\draw (10,1) -- (10,2);
\draw (10,3) -- (10,4);
\draw (10,5) -- (10,6);
\draw (10,7) -- (10,8);
\draw (10,9) -- (10,10);
\draw (10,11) -- (10,12);
\draw (11,-1) -- (11,0);
\draw (11,1) -- (11,2);
\draw (11,3) -- (11,4);
\draw (11,5) -- (11,6);
\draw (11,7) -- (11,8);
\draw (11,9) -- (11,10);
\draw (11,11) -- (11,12);
\draw[red, very thick, dashed] (-0.45,-0.45) rectangle (3.45,6.45);
\fill (-4,-2) circle (5pt);
\fill[white] (-4,-1) circle (2.6pt); \draw (-4,-1) circle (2.6pt);
\fill[white] (-4,0) circle (2.6pt); \draw (-4,0) circle (2.6pt);
\fill[white] (-4,1) circle (2.6pt); \draw (-4,1) circle (2.6pt);
\fill[white] (-4,2) circle (2.6pt); \draw (-4,2) circle (2.6pt);
\fill[white] (-4,3) circle (2.6pt); \draw (-4,3) circle (2.6pt);
\fill[white] (-4,4) circle (2.6pt); \draw (-4,4) circle (2.6pt);
\fill[white] (-4,5) circle (2.6pt); \draw (-4,5) circle (2.6pt);
\fill[white] (-4,6) circle (2.6pt); \draw (-4,6) circle (2.6pt);
\fill (-4,7) circle (5pt);
\fill[white] (-4,8) circle (2.6pt); \draw (-4,8) circle (2.6pt);
\fill (-4,9) circle (5pt);
\fill (-4,10) circle (5pt);
\fill[white] (-4,11) circle (2.6pt); \draw (-4,11) circle (2.6pt);
\fill (-4,12) circle (5pt);
\fill (-3,-2) circle (5pt);
\fill[white] (-3,-1) circle (2.6pt); \draw (-3,-1) circle (2.6pt);
\fill (-3,0) circle (5pt);
\fill[white] (-3,1) circle (2.6pt); \draw (-3,1) circle (2.6pt);
\fill[white] (-3,2) circle (2.6pt); \draw (-3,2) circle (2.6pt);
\fill[white] (-3,3) circle (2.6pt); \draw (-3,3) circle (2.6pt);
\fill[white] (-3,4) circle (2.6pt); \draw (-3,4) circle (2.6pt);
\fill (-3,5) circle (5pt);
\fill[white] (-3,6) circle (2.6pt); \draw (-3,6) circle (2.6pt);
\fill (-3,7) circle (5pt);
\fill (-3,8) circle (5pt);
\fill (-3,9) circle (5pt);
\fill (-3,10) circle (5pt);
\fill (-3,11) circle (5pt);
\fill (-3,12) circle (5pt);
\fill[white] (-2,-2) circle (2.6pt); \draw (-2,-2) circle (2.6pt);
\fill[white] (-2,-1) circle (2.6pt); \draw (-2,-1) circle (2.6pt);
\fill (-2,0) circle (5pt);
\fill[white] (-2,1) circle (2.6pt); \draw (-2,1) circle (2.6pt);
\fill (-2,2) circle (5pt);
\fill (-2,3) circle (5pt);
\fill[white] (-2,4) circle (2.6pt); \draw (-2,4) circle (2.6pt);
\fill (-2,5) circle (5pt);
\fill[white] (-2,6) circle (2.6pt); \draw (-2,6) circle (2.6pt);
\fill[white] (-2,7) circle (2.6pt); \draw (-2,7) circle (2.6pt);
\fill[white] (-2,8) circle (2.6pt); \draw (-2,8) circle (2.6pt);
\fill[white] (-2,9) circle (2.6pt); \draw (-2,9) circle (2.6pt);
\fill[white] (-2,10) circle (2.6pt); \draw (-2,10) circle (2.6pt);
\fill[white] (-2,11) circle (2.6pt); \draw (-2,11) circle (2.6pt);
\fill[white] (-2,12) circle (2.6pt); \draw (-2,12) circle (2.6pt);
\fill (-1,-2) circle (5pt);
\fill[white] (-1,-1) circle (2.6pt); \draw (-1,-1) circle (2.6pt);
\fill (-1,0) circle (5pt);
\fill (-1,1) circle (5pt);
\fill (-1,2) circle (5pt);
\fill (-1,3) circle (5pt);
\fill (-1,4) circle (5pt);
\fill (-1,5) circle (5pt);
\fill[white] (-1,6) circle (2.6pt); \draw (-1,6) circle (2.6pt);
\fill (-1,7) circle (5pt);
\fill[white] (-1,8) circle (2.6pt); \draw (-1,8) circle (2.6pt);
\fill[white] (-1,9) circle (2.6pt); \draw (-1,9) circle (2.6pt);
\fill[white] (-1,10) circle (2.6pt); \draw (-1,10) circle (2.6pt);
\fill[white] (-1,11) circle (2.6pt); \draw (-1,11) circle (2.6pt);
\fill (-1,12) circle (5pt);
\fill (0,-2) circle (5pt);
\fill[white] (0,-1) circle (2.6pt); \draw (0,-1) circle (2.6pt);
\fill[white] (0,0) circle (2.6pt); \draw (0,0) circle (2.6pt);
\fill[white] (0,1) circle (2.6pt); \draw (0,1) circle (2.6pt);
\fill[white] (0,2) circle (2.6pt); \draw (0,2) circle (2.6pt);
\fill[white] (0,3) circle (2.6pt); \draw (0,3) circle (2.6pt);
\fill[white] (0,4) circle (2.6pt); \draw (0,4) circle (2.6pt);
\fill[white] (0,5) circle (2.6pt); \draw (0,5) circle (2.6pt);
\fill[white] (0,6) circle (2.6pt); \draw (0,6) circle (2.6pt);
\fill (0,7) circle (5pt);
\fill[white] (0,8) circle (2.6pt); \draw (0,8) circle (2.6pt);
\fill (0,9) circle (5pt);
\fill (0,10) circle (5pt);
\fill[white] (0,11) circle (2.6pt); \draw (0,11) circle (2.6pt);
\fill (0,12) circle (5pt);
\fill (1,-2) circle (5pt);
\fill[white] (1,-1) circle (2.6pt); \draw (1,-1) circle (2.6pt);
\fill (1,0) circle (5pt);
\fill[white] (1,1) circle (2.6pt); \draw (1,1) circle (2.6pt);
\fill[white] (1,2) circle (2.6pt); \draw (1,2) circle (2.6pt);
\fill[white] (1,3) circle (2.6pt); \draw (1,3) circle (2.6pt);
\fill[white] (1,4) circle (2.6pt); \draw (1,4) circle (2.6pt);
\fill (1,5) circle (5pt);
\fill[white] (1,6) circle (2.6pt); \draw (1,6) circle (2.6pt);
\fill (1,7) circle (5pt);
\fill (1,8) circle (5pt);
\fill (1,9) circle (5pt);
\fill (1,10) circle (5pt);
\fill (1,11) circle (5pt);
\fill (1,12) circle (5pt);
\fill[white] (2,-2) circle (2.6pt); \draw (2,-2) circle (2.6pt);
\fill[white] (2,-1) circle (2.6pt); \draw (2,-1) circle (2.6pt);
\fill (2,0) circle (5pt);
\fill[white] (2,1) circle (2.6pt); \draw (2,1) circle (2.6pt);
\fill (2,2) circle (5pt);
\fill (2,3) circle (5pt);
\fill[white] (2,4) circle (2.6pt); \draw (2,4) circle (2.6pt);
\fill (2,5) circle (5pt);
\fill[white] (2,6) circle (2.6pt); \draw (2,6) circle (2.6pt);
\fill[white] (2,7) circle (2.6pt); \draw (2,7) circle (2.6pt);
\fill[white] (2,8) circle (2.6pt); \draw (2,8) circle (2.6pt);
\fill[white] (2,9) circle (2.6pt); \draw (2,9) circle (2.6pt);
\fill[white] (2,10) circle (2.6pt); \draw (2,10) circle (2.6pt);
\fill[white] (2,11) circle (2.6pt); \draw (2,11) circle (2.6pt);
\fill[white] (2,12) circle (2.6pt); \draw (2,12) circle (2.6pt);
\fill (3,-2) circle (5pt);
\fill[white] (3,-1) circle (2.6pt); \draw (3,-1) circle (2.6pt);
\fill (3,0) circle (5pt);
\fill (3,1) circle (5pt);
\fill (3,2) circle (5pt);
\fill (3,3) circle (5pt);
\fill (3,4) circle (5pt);
\fill (3,5) circle (5pt);
\fill[white] (3,6) circle (2.6pt); \draw (3,6) circle (2.6pt);
\fill (3,7) circle (5pt);
\fill[white] (3,8) circle (2.6pt); \draw (3,8) circle (2.6pt);
\fill[white] (3,9) circle (2.6pt); \draw (3,9) circle (2.6pt);
\fill[white] (3,10) circle (2.6pt); \draw (3,10) circle (2.6pt);
\fill[white] (3,11) circle (2.6pt); \draw (3,11) circle (2.6pt);
\fill (3,12) circle (5pt);
\fill (4,-2) circle (5pt);
\fill[white] (4,-1) circle (2.6pt); \draw (4,-1) circle (2.6pt);
\fill[white] (4,0) circle (2.6pt); \draw (4,0) circle (2.6pt);
\fill[white] (4,1) circle (2.6pt); \draw (4,1) circle (2.6pt);
\fill[white] (4,2) circle (2.6pt); \draw (4,2) circle (2.6pt);
\fill[white] (4,3) circle (2.6pt); \draw (4,3) circle (2.6pt);
\fill[white] (4,4) circle (2.6pt); \draw (4,4) circle (2.6pt);
\fill[white] (4,5) circle (2.6pt); \draw (4,5) circle (2.6pt);
\fill[white] (4,6) circle (2.6pt); \draw (4,6) circle (2.6pt);
\fill (4,7) circle (5pt);
\fill[white] (4,8) circle (2.6pt); \draw (4,8) circle (2.6pt);
\fill (4,9) circle (5pt);
\fill (4,10) circle (5pt);
\fill[white] (4,11) circle (2.6pt); \draw (4,11) circle (2.6pt);
\fill (4,12) circle (5pt);
\fill (5,-2) circle (5pt);
\fill[white] (5,-1) circle (2.6pt); \draw (5,-1) circle (2.6pt);
\fill (5,0) circle (5pt);
\fill[white] (5,1) circle (2.6pt); \draw (5,1) circle (2.6pt);
\fill[white] (5,2) circle (2.6pt); \draw (5,2) circle (2.6pt);
\fill[white] (5,3) circle (2.6pt); \draw (5,3) circle (2.6pt);
\fill[white] (5,4) circle (2.6pt); \draw (5,4) circle (2.6pt);
\fill (5,5) circle (5pt);
\fill[white] (5,6) circle (2.6pt); \draw (5,6) circle (2.6pt);
\fill (5,7) circle (5pt);
\fill (5,8) circle (5pt);
\fill (5,9) circle (5pt);
\fill (5,10) circle (5pt);
\fill (5,11) circle (5pt);
\fill (5,12) circle (5pt);
\fill[white] (6,-2) circle (2.6pt); \draw (6,-2) circle (2.6pt);
\fill[white] (6,-1) circle (2.6pt); \draw (6,-1) circle (2.6pt);
\fill (6,0) circle (5pt);
\fill[white] (6,1) circle (2.6pt); \draw (6,1) circle (2.6pt);
\fill (6,2) circle (5pt);
\fill (6,3) circle (5pt);
\fill[white] (6,4) circle (2.6pt); \draw (6,4) circle (2.6pt);
\fill (6,5) circle (5pt);
\fill[white] (6,6) circle (2.6pt); \draw (6,6) circle (2.6pt);
\fill[white] (6,7) circle (2.6pt); \draw (6,7) circle (2.6pt);
\fill[white] (6,8) circle (2.6pt); \draw (6,8) circle (2.6pt);
\fill[white] (6,9) circle (2.6pt); \draw (6,9) circle (2.6pt);
\fill[white] (6,10) circle (2.6pt); \draw (6,10) circle (2.6pt);
\fill[white] (6,11) circle (2.6pt); \draw (6,11) circle (2.6pt);
\fill[white] (6,12) circle (2.6pt); \draw (6,12) circle (2.6pt);
\fill (7,-2) circle (5pt);
\fill[white] (7,-1) circle (2.6pt); \draw (7,-1) circle (2.6pt);
\fill (7,0) circle (5pt);
\fill (7,1) circle (5pt);
\fill (7,2) circle (5pt);
\fill (7,3) circle (5pt);
\fill (7,4) circle (5pt);
\fill (7,5) circle (5pt);
\fill[white] (7,6) circle (2.6pt); \draw (7,6) circle (2.6pt);
\fill (7,7) circle (5pt);
\fill[white] (7,8) circle (2.6pt); \draw (7,8) circle (2.6pt);
\fill[white] (7,9) circle (2.6pt); \draw (7,9) circle (2.6pt);
\fill[white] (7,10) circle (2.6pt); \draw (7,10) circle (2.6pt);
\fill[white] (7,11) circle (2.6pt); \draw (7,11) circle (2.6pt);
\fill (7,12) circle (5pt);
\fill (8,-2) circle (5pt);
\fill[white] (8,-1) circle (2.6pt); \draw (8,-1) circle (2.6pt);
\fill[white] (8,0) circle (2.6pt); \draw (8,0) circle (2.6pt);
\fill[white] (8,1) circle (2.6pt); \draw (8,1) circle (2.6pt);
\fill[white] (8,2) circle (2.6pt); \draw (8,2) circle (2.6pt);
\fill[white] (8,3) circle (2.6pt); \draw (8,3) circle (2.6pt);
\fill[white] (8,4) circle (2.6pt); \draw (8,4) circle (2.6pt);
\fill[white] (8,5) circle (2.6pt); \draw (8,5) circle (2.6pt);
\fill[white] (8,6) circle (2.6pt); \draw (8,6) circle (2.6pt);
\fill (8,7) circle (5pt);
\fill[white] (8,8) circle (2.6pt); \draw (8,8) circle (2.6pt);
\fill (8,9) circle (5pt);
\fill (8,10) circle (5pt);
\fill[white] (8,11) circle (2.6pt); \draw (8,11) circle (2.6pt);
\fill (8,12) circle (5pt);
\fill (9,-2) circle (5pt);
\fill[white] (9,-1) circle (2.6pt); \draw (9,-1) circle (2.6pt);
\fill (9,0) circle (5pt);
\fill[white] (9,1) circle (2.6pt); \draw (9,1) circle (2.6pt);
\fill[white] (9,2) circle (2.6pt); \draw (9,2) circle (2.6pt);
\fill[white] (9,3) circle (2.6pt); \draw (9,3) circle (2.6pt);
\fill[white] (9,4) circle (2.6pt); \draw (9,4) circle (2.6pt);
\fill (9,5) circle (5pt);
\fill[white] (9,6) circle (2.6pt); \draw (9,6) circle (2.6pt);
\fill (9,7) circle (5pt);
\fill (9,8) circle (5pt);
\fill (9,9) circle (5pt);
\fill (9,10) circle (5pt);
\fill (9,11) circle (5pt);
\fill (9,12) circle (5pt);
\fill[white] (10,-2) circle (2.6pt); \draw (10,-2) circle (2.6pt);
\fill[white] (10,-1) circle (2.6pt); \draw (10,-1) circle (2.6pt);
\fill (10,0) circle (5pt);
\fill[white] (10,1) circle (2.6pt); \draw (10,1) circle (2.6pt);
\fill (10,2) circle (5pt);
\fill (10,3) circle (5pt);
\fill[white] (10,4) circle (2.6pt); \draw (10,4) circle (2.6pt);
\fill (10,5) circle (5pt);
\fill[white] (10,6) circle (2.6pt); \draw (10,6) circle (2.6pt);
\fill[white] (10,7) circle (2.6pt); \draw (10,7) circle (2.6pt);
\fill[white] (10,8) circle (2.6pt); \draw (10,8) circle (2.6pt);
\fill[white] (10,9) circle (2.6pt); \draw (10,9) circle (2.6pt);
\fill[white] (10,10) circle (2.6pt); \draw (10,10) circle (2.6pt);
\fill[white] (10,11) circle (2.6pt); \draw (10,11) circle (2.6pt);
\fill[white] (10,12) circle (2.6pt); \draw (10,12) circle (2.6pt);
\fill (11,-2) circle (5pt);
\fill[white] (11,-1) circle (2.6pt); \draw (11,-1) circle (2.6pt);
\fill (11,0) circle (5pt);
\fill (11,1) circle (5pt);
\fill (11,2) circle (5pt);
\fill (11,3) circle (5pt);
\fill (11,4) circle (5pt);
\fill (11,5) circle (5pt);
\fill[white] (11,6) circle (2.6pt); \draw (11,6) circle (2.6pt);
\fill (11,7) circle (5pt);
\fill[white] (11,8) circle (2.6pt); \draw (11,8) circle (2.6pt);
\fill[white] (11,9) circle (2.6pt); \draw (11,9) circle (2.6pt);
\fill[white] (11,10) circle (2.6pt); \draw (11,10) circle (2.6pt);
\fill[white] (11,11) circle (2.6pt); \draw (11,11) circle (2.6pt);
\fill (11,12) circle (5pt);
\end{tikzpicture}
\caption{An idcode of density $3/7$ of the truncated square grid $4.8.8$ (black vertices), drawn as a brick wall: the shaded $4$-cycles are the squares and all other faces are octagons. The code is periodic, and the dashed rectangle is a fundamental domain of its period lattice: it contains $28$ vertices, $12$ of which are codewords.}
\label{fig:code}
\end{figure}

\begin{theorem}\label{thm:tsq}
The minimum density of an idcode of the truncated square grid $4.8.8$ is $3/7$.
\end{theorem}

\begin{proof}
The truncated square grid is twin-free and cubic, and every vertex lies on exactly one $4$-cycle, so Theorem~\ref{thm:4cycle} gives the lower bound. Conversely, the black vertices of Figure~\ref{fig:code} form a periodic set of density $12/28=3/7$. Since two vertices at distance at least $3$ have disjoint closed neighborhoods, by periodicity it suffices to check that $I(u)\neq\emptyset$ and $I(u)\neq I(v)$ for every vertex $u$ of the dashed rectangle and every vertex $v$ with $d(u,v)\le 2$, which is a routine verification with Figure~\ref{fig:code}.
\end{proof}

Every codeword of the code of Figure~\ref{fig:code} ends with charge exactly $3$, so this code is perfectly balanced for the rules (R1) and (R2). Note also that the proof of Theorem~\ref{thm:4cycle} only uses that the four vertices of $N[u]$ are pairwise separated and that $C$ is dominating. It would be interesting to know whether rules of this kind give good bounds for the hexagonal grid, which has a structure similar to the truncated square grid.

\section{The elongated triangular grid}\label{sec:elong}

The elongated triangular tiling has vertex configuration $3^3.4^2$ (see Figure~\ref{fig:elong}). Its graph is $5$-regular, so the general lower bound of~\cite{SSW24} gives only $2/7$. In this section we prove that the minimum density is $1/3$, which is item~(d) of Theorem~\ref{thm:main}.

The horizontal lines of the tiling bound alternately a row of squares and a row of triangles. We draw the grid as the square grid in which the squares of alternate rows are cut by a diagonal at $45^\circ$, so that these rows become the rows of triangles. Every edge lies on two faces, and there are three kinds of edges: the vertical edges of the rows of squares lie on two squares, the horizontal edges lie on a square and a triangle, and the other edges, in the rows of triangles, lie on two triangles. Accordingly, every vertex $v$ has exactly one neighbor joined to it by an edge lying on two squares, which we call its \emph{bridge neighbor} $\beta(v)$, two \emph{outer} neighbors, which are its neighbors on its row, and two \emph{inner} neighbors, joined to it by edges lying on two triangles. Since the kind of an edge does not depend on the endpoint, $\beta(\beta(v))=v$, and $v$ is an outer (inner) neighbor of each of its outer (inner) neighbors. Around $v$ the faces are two squares and three triangles, so the neighbors of $v$ appear in the cyclic order $\beta(v)$, outer, inner, inner, outer, and the three triangles containing $v$ are formed by $v$ with each consecutive pair among the last four. Hence the two inner neighbors of $v$ are adjacent, each outer neighbor of $v$ is adjacent to exactly one inner neighbor of $v$, and these are the only edges among the neighbors of $v$.

\begin{figure}[t]
\centering
\begin{tikzpicture}[scale=0.496]
\begin{scope}[shift={{(0,0)}}]
\draw (0,0)--(1,0);
\draw (1,0)--(2,0);
\draw (2,0)--(3,0);
\draw (3,0)--(4,0);
\draw (4,0)--(5,0);
\draw (5,0)--(6,0);
\draw (6,0)--(7,0);
\draw (7,0)--(8,0);
\draw (0,1)--(1,1);
\draw (1,1)--(2,1);
\draw (2,1)--(3,1);
\draw (3,1)--(4,1);
\draw (4,1)--(5,1);
\draw (5,1)--(6,1);
\draw (6,1)--(7,1);
\draw (7,1)--(8,1);
\draw (0,2)--(1,2);
\draw (1,2)--(2,2);
\draw (2,2)--(3,2);
\draw (3,2)--(4,2);
\draw (4,2)--(5,2);
\draw (5,2)--(6,2);
\draw (6,2)--(7,2);
\draw (7,2)--(8,2);
\draw (0,3)--(1,3);
\draw (1,3)--(2,3);
\draw (2,3)--(3,3);
\draw (3,3)--(4,3);
\draw (4,3)--(5,3);
\draw (5,3)--(6,3);
\draw (6,3)--(7,3);
\draw (7,3)--(8,3);
\draw (0,4)--(1,4);
\draw (1,4)--(2,4);
\draw (2,4)--(3,4);
\draw (3,4)--(4,4);
\draw (4,4)--(5,4);
\draw (5,4)--(6,4);
\draw (6,4)--(7,4);
\draw (7,4)--(8,4);
\draw (0,5)--(1,5);
\draw (1,5)--(2,5);
\draw (2,5)--(3,5);
\draw (3,5)--(4,5);
\draw (4,5)--(5,5);
\draw (5,5)--(6,5);
\draw (6,5)--(7,5);
\draw (7,5)--(8,5);
\draw (0,6)--(1,6);
\draw (1,6)--(2,6);
\draw (2,6)--(3,6);
\draw (3,6)--(4,6);
\draw (4,6)--(5,6);
\draw (5,6)--(6,6);
\draw (6,6)--(7,6);
\draw (7,6)--(8,6);
\draw (0,7)--(1,7);
\draw (1,7)--(2,7);
\draw (2,7)--(3,7);
\draw (3,7)--(4,7);
\draw (4,7)--(5,7);
\draw (5,7)--(6,7);
\draw (6,7)--(7,7);
\draw (7,7)--(8,7);
\draw (0,0)--(0,1);
\draw (0,1)--(0,2);
\draw (0,2)--(0,3);
\draw (0,3)--(0,4);
\draw (0,4)--(0,5);
\draw (0,5)--(0,6);
\draw (0,6)--(0,7);
\draw (1,0)--(1,1);
\draw (1,1)--(1,2);
\draw (1,2)--(1,3);
\draw (1,3)--(1,4);
\draw (1,4)--(1,5);
\draw (1,5)--(1,6);
\draw (1,6)--(1,7);
\draw (2,0)--(2,1);
\draw (2,1)--(2,2);
\draw (2,2)--(2,3);
\draw (2,3)--(2,4);
\draw (2,4)--(2,5);
\draw (2,5)--(2,6);
\draw (2,6)--(2,7);
\draw (3,0)--(3,1);
\draw (3,1)--(3,2);
\draw (3,2)--(3,3);
\draw (3,3)--(3,4);
\draw (3,4)--(3,5);
\draw (3,5)--(3,6);
\draw (3,6)--(3,7);
\draw (4,0)--(4,1);
\draw (4,1)--(4,2);
\draw (4,2)--(4,3);
\draw (4,3)--(4,4);
\draw (4,4)--(4,5);
\draw (4,5)--(4,6);
\draw (4,6)--(4,7);
\draw (5,0)--(5,1);
\draw (5,1)--(5,2);
\draw (5,2)--(5,3);
\draw (5,3)--(5,4);
\draw (5,4)--(5,5);
\draw (5,5)--(5,6);
\draw (5,6)--(5,7);
\draw (6,0)--(6,1);
\draw (6,1)--(6,2);
\draw (6,2)--(6,3);
\draw (6,3)--(6,4);
\draw (6,4)--(6,5);
\draw (6,5)--(6,6);
\draw (6,6)--(6,7);
\draw (7,0)--(7,1);
\draw (7,1)--(7,2);
\draw (7,2)--(7,3);
\draw (7,3)--(7,4);
\draw (7,4)--(7,5);
\draw (7,5)--(7,6);
\draw (7,6)--(7,7);
\draw (8,0)--(8,1);
\draw (8,1)--(8,2);
\draw (8,2)--(8,3);
\draw (8,3)--(8,4);
\draw (8,4)--(8,5);
\draw (8,5)--(8,6);
\draw (8,6)--(8,7);
\draw (0,1)--(1,2);
\draw (1,1)--(2,2);
\draw (2,1)--(3,2);
\draw (3,1)--(4,2);
\draw (4,1)--(5,2);
\draw (5,1)--(6,2);
\draw (6,1)--(7,2);
\draw (7,1)--(8,2);
\draw (0,3)--(1,4);
\draw (1,3)--(2,4);
\draw (2,3)--(3,4);
\draw (3,3)--(4,4);
\draw (4,3)--(5,4);
\draw (5,3)--(6,4);
\draw (6,3)--(7,4);
\draw (7,3)--(8,4);
\draw (0,5)--(1,6);
\draw (1,5)--(2,6);
\draw (2,5)--(3,6);
\draw (3,5)--(4,6);
\draw (4,5)--(5,6);
\draw (5,5)--(6,6);
\draw (6,5)--(7,6);
\draw (7,5)--(8,6);
\fill (0,0) circle (4.2pt);
\fill (1,0) circle (4.2pt);
\fill (2,0) circle (4.2pt);
\fill (3,0) circle (4.2pt);
\fill (4,0) circle (4.2pt);
\fill (5,0) circle (4.2pt);
\fill (6,0) circle (4.2pt);
\fill (7,0) circle (4.2pt);
\fill (8,0) circle (4.2pt);
\fill[white] (0,1) circle (3.0pt); \draw (0,1) circle (3.0pt);
\fill[white] (1,1) circle (3.0pt); \draw (1,1) circle (3.0pt);
\fill[white] (2,1) circle (3.0pt); \draw (2,1) circle (3.0pt);
\fill[white] (3,1) circle (3.0pt); \draw (3,1) circle (3.0pt);
\fill[white] (4,1) circle (3.0pt); \draw (4,1) circle (3.0pt);
\fill[white] (5,1) circle (3.0pt); \draw (5,1) circle (3.0pt);
\fill[white] (6,1) circle (3.0pt); \draw (6,1) circle (3.0pt);
\fill[white] (7,1) circle (3.0pt); \draw (7,1) circle (3.0pt);
\fill[white] (8,1) circle (3.0pt); \draw (8,1) circle (3.0pt);
\fill[white] (0,2) circle (3.0pt); \draw (0,2) circle (3.0pt);
\fill[white] (1,2) circle (3.0pt); \draw (1,2) circle (3.0pt);
\fill[white] (2,2) circle (3.0pt); \draw (2,2) circle (3.0pt);
\fill[white] (3,2) circle (3.0pt); \draw (3,2) circle (3.0pt);
\fill[white] (4,2) circle (3.0pt); \draw (4,2) circle (3.0pt);
\fill[white] (5,2) circle (3.0pt); \draw (5,2) circle (3.0pt);
\fill[white] (6,2) circle (3.0pt); \draw (6,2) circle (3.0pt);
\fill[white] (7,2) circle (3.0pt); \draw (7,2) circle (3.0pt);
\fill[white] (8,2) circle (3.0pt); \draw (8,2) circle (3.0pt);
\fill (0,3) circle (4.2pt);
\fill (1,3) circle (4.2pt);
\fill (2,3) circle (4.2pt);
\fill (3,3) circle (4.2pt);
\fill (4,3) circle (4.2pt);
\fill (5,3) circle (4.2pt);
\fill (6,3) circle (4.2pt);
\fill (7,3) circle (4.2pt);
\fill (8,3) circle (4.2pt);
\fill[white] (0,4) circle (3.0pt); \draw (0,4) circle (3.0pt);
\fill[white] (1,4) circle (3.0pt); \draw (1,4) circle (3.0pt);
\fill[white] (2,4) circle (3.0pt); \draw (2,4) circle (3.0pt);
\fill[white] (3,4) circle (3.0pt); \draw (3,4) circle (3.0pt);
\fill[white] (4,4) circle (3.0pt); \draw (4,4) circle (3.0pt);
\fill[white] (5,4) circle (3.0pt); \draw (5,4) circle (3.0pt);
\fill[white] (6,4) circle (3.0pt); \draw (6,4) circle (3.0pt);
\fill[white] (7,4) circle (3.0pt); \draw (7,4) circle (3.0pt);
\fill[white] (8,4) circle (3.0pt); \draw (8,4) circle (3.0pt);
\fill[white] (0,5) circle (3.0pt); \draw (0,5) circle (3.0pt);
\fill[white] (1,5) circle (3.0pt); \draw (1,5) circle (3.0pt);
\fill[white] (2,5) circle (3.0pt); \draw (2,5) circle (3.0pt);
\fill[white] (3,5) circle (3.0pt); \draw (3,5) circle (3.0pt);
\fill[white] (4,5) circle (3.0pt); \draw (4,5) circle (3.0pt);
\fill[white] (5,5) circle (3.0pt); \draw (5,5) circle (3.0pt);
\fill[white] (6,5) circle (3.0pt); \draw (6,5) circle (3.0pt);
\fill[white] (7,5) circle (3.0pt); \draw (7,5) circle (3.0pt);
\fill[white] (8,5) circle (3.0pt); \draw (8,5) circle (3.0pt);
\fill (0,6) circle (4.2pt);
\fill (1,6) circle (4.2pt);
\fill (2,6) circle (4.2pt);
\fill (3,6) circle (4.2pt);
\fill (4,6) circle (4.2pt);
\fill (5,6) circle (4.2pt);
\fill (6,6) circle (4.2pt);
\fill (7,6) circle (4.2pt);
\fill (8,6) circle (4.2pt);
\fill[white] (0,7) circle (3.0pt); \draw (0,7) circle (3.0pt);
\fill[white] (1,7) circle (3.0pt); \draw (1,7) circle (3.0pt);
\fill[white] (2,7) circle (3.0pt); \draw (2,7) circle (3.0pt);
\fill[white] (3,7) circle (3.0pt); \draw (3,7) circle (3.0pt);
\fill[white] (4,7) circle (3.0pt); \draw (4,7) circle (3.0pt);
\fill[white] (5,7) circle (3.0pt); \draw (5,7) circle (3.0pt);
\fill[white] (6,7) circle (3.0pt); \draw (6,7) circle (3.0pt);
\fill[white] (7,7) circle (3.0pt); \draw (7,7) circle (3.0pt);
\fill[white] (8,7) circle (3.0pt); \draw (8,7) circle (3.0pt);
\node at (4.00,-0.9) {(a)};
\end{scope}
\begin{scope}[shift={{(9.5,0)}}]
\draw (0,0)--(1,0);
\draw (1,0)--(2,0);
\draw (2,0)--(3,0);
\draw (3,0)--(4,0);
\draw (4,0)--(5,0);
\draw (5,0)--(6,0);
\draw (6,0)--(7,0);
\draw (7,0)--(8,0);
\draw (0,1)--(1,1);
\draw (1,1)--(2,1);
\draw (2,1)--(3,1);
\draw (3,1)--(4,1);
\draw (4,1)--(5,1);
\draw (5,1)--(6,1);
\draw (6,1)--(7,1);
\draw (7,1)--(8,1);
\draw (0,2)--(1,2);
\draw (1,2)--(2,2);
\draw (2,2)--(3,2);
\draw (3,2)--(4,2);
\draw (4,2)--(5,2);
\draw (5,2)--(6,2);
\draw (6,2)--(7,2);
\draw (7,2)--(8,2);
\draw (0,3)--(1,3);
\draw (1,3)--(2,3);
\draw (2,3)--(3,3);
\draw (3,3)--(4,3);
\draw (4,3)--(5,3);
\draw (5,3)--(6,3);
\draw (6,3)--(7,3);
\draw (7,3)--(8,3);
\draw (0,4)--(1,4);
\draw (1,4)--(2,4);
\draw (2,4)--(3,4);
\draw (3,4)--(4,4);
\draw (4,4)--(5,4);
\draw (5,4)--(6,4);
\draw (6,4)--(7,4);
\draw (7,4)--(8,4);
\draw (0,5)--(1,5);
\draw (1,5)--(2,5);
\draw (2,5)--(3,5);
\draw (3,5)--(4,5);
\draw (4,5)--(5,5);
\draw (5,5)--(6,5);
\draw (6,5)--(7,5);
\draw (7,5)--(8,5);
\draw (0,6)--(1,6);
\draw (1,6)--(2,6);
\draw (2,6)--(3,6);
\draw (3,6)--(4,6);
\draw (4,6)--(5,6);
\draw (5,6)--(6,6);
\draw (6,6)--(7,6);
\draw (7,6)--(8,6);
\draw (0,7)--(1,7);
\draw (1,7)--(2,7);
\draw (2,7)--(3,7);
\draw (3,7)--(4,7);
\draw (4,7)--(5,7);
\draw (5,7)--(6,7);
\draw (6,7)--(7,7);
\draw (7,7)--(8,7);
\draw (0,0)--(0,1);
\draw (0,1)--(0,2);
\draw (0,2)--(0,3);
\draw (0,3)--(0,4);
\draw (0,4)--(0,5);
\draw (0,5)--(0,6);
\draw (0,6)--(0,7);
\draw (1,0)--(1,1);
\draw (1,1)--(1,2);
\draw (1,2)--(1,3);
\draw (1,3)--(1,4);
\draw (1,4)--(1,5);
\draw (1,5)--(1,6);
\draw (1,6)--(1,7);
\draw (2,0)--(2,1);
\draw (2,1)--(2,2);
\draw (2,2)--(2,3);
\draw (2,3)--(2,4);
\draw (2,4)--(2,5);
\draw (2,5)--(2,6);
\draw (2,6)--(2,7);
\draw (3,0)--(3,1);
\draw (3,1)--(3,2);
\draw (3,2)--(3,3);
\draw (3,3)--(3,4);
\draw (3,4)--(3,5);
\draw (3,5)--(3,6);
\draw (3,6)--(3,7);
\draw (4,0)--(4,1);
\draw (4,1)--(4,2);
\draw (4,2)--(4,3);
\draw (4,3)--(4,4);
\draw (4,4)--(4,5);
\draw (4,5)--(4,6);
\draw (4,6)--(4,7);
\draw (5,0)--(5,1);
\draw (5,1)--(5,2);
\draw (5,2)--(5,3);
\draw (5,3)--(5,4);
\draw (5,4)--(5,5);
\draw (5,5)--(5,6);
\draw (5,6)--(5,7);
\draw (6,0)--(6,1);
\draw (6,1)--(6,2);
\draw (6,2)--(6,3);
\draw (6,3)--(6,4);
\draw (6,4)--(6,5);
\draw (6,5)--(6,6);
\draw (6,6)--(6,7);
\draw (7,0)--(7,1);
\draw (7,1)--(7,2);
\draw (7,2)--(7,3);
\draw (7,3)--(7,4);
\draw (7,4)--(7,5);
\draw (7,5)--(7,6);
\draw (7,6)--(7,7);
\draw (8,0)--(8,1);
\draw (8,1)--(8,2);
\draw (8,2)--(8,3);
\draw (8,3)--(8,4);
\draw (8,4)--(8,5);
\draw (8,5)--(8,6);
\draw (8,6)--(8,7);
\draw (0,1)--(1,2);
\draw (1,1)--(2,2);
\draw (2,1)--(3,2);
\draw (3,1)--(4,2);
\draw (4,1)--(5,2);
\draw (5,1)--(6,2);
\draw (6,1)--(7,2);
\draw (7,1)--(8,2);
\draw (0,3)--(1,4);
\draw (1,3)--(2,4);
\draw (2,3)--(3,4);
\draw (3,3)--(4,4);
\draw (4,3)--(5,4);
\draw (5,3)--(6,4);
\draw (6,3)--(7,4);
\draw (7,3)--(8,4);
\draw (0,5)--(1,6);
\draw (1,5)--(2,6);
\draw (2,5)--(3,6);
\draw (3,5)--(4,6);
\draw (4,5)--(5,6);
\draw (5,5)--(6,6);
\draw (6,5)--(7,6);
\draw (7,5)--(8,6);
\fill (0,0) circle (4.2pt);
\fill (1,0) circle (4.2pt);
\fill[white] (2,0) circle (3.0pt); \draw (2,0) circle (3.0pt);
\fill (3,0) circle (4.2pt);
\fill (4,0) circle (4.2pt);
\fill[white] (5,0) circle (3.0pt); \draw (5,0) circle (3.0pt);
\fill (6,0) circle (4.2pt);
\fill (7,0) circle (4.2pt);
\fill[white] (8,0) circle (3.0pt); \draw (8,0) circle (3.0pt);
\fill (0,1) circle (4.2pt);
\fill (1,1) circle (4.2pt);
\fill[white] (2,1) circle (3.0pt); \draw (2,1) circle (3.0pt);
\fill (3,1) circle (4.2pt);
\fill (4,1) circle (4.2pt);
\fill[white] (5,1) circle (3.0pt); \draw (5,1) circle (3.0pt);
\fill (6,1) circle (4.2pt);
\fill (7,1) circle (4.2pt);
\fill[white] (8,1) circle (3.0pt); \draw (8,1) circle (3.0pt);
\fill[white] (0,2) circle (3.0pt); \draw (0,2) circle (3.0pt);
\fill[white] (1,2) circle (3.0pt); \draw (1,2) circle (3.0pt);
\fill[white] (2,2) circle (3.0pt); \draw (2,2) circle (3.0pt);
\fill[white] (3,2) circle (3.0pt); \draw (3,2) circle (3.0pt);
\fill[white] (4,2) circle (3.0pt); \draw (4,2) circle (3.0pt);
\fill[white] (5,2) circle (3.0pt); \draw (5,2) circle (3.0pt);
\fill[white] (6,2) circle (3.0pt); \draw (6,2) circle (3.0pt);
\fill[white] (7,2) circle (3.0pt); \draw (7,2) circle (3.0pt);
\fill[white] (8,2) circle (3.0pt); \draw (8,2) circle (3.0pt);
\fill[white] (0,3) circle (3.0pt); \draw (0,3) circle (3.0pt);
\fill[white] (1,3) circle (3.0pt); \draw (1,3) circle (3.0pt);
\fill[white] (2,3) circle (3.0pt); \draw (2,3) circle (3.0pt);
\fill[white] (3,3) circle (3.0pt); \draw (3,3) circle (3.0pt);
\fill[white] (4,3) circle (3.0pt); \draw (4,3) circle (3.0pt);
\fill[white] (5,3) circle (3.0pt); \draw (5,3) circle (3.0pt);
\fill[white] (6,3) circle (3.0pt); \draw (6,3) circle (3.0pt);
\fill[white] (7,3) circle (3.0pt); \draw (7,3) circle (3.0pt);
\fill[white] (8,3) circle (3.0pt); \draw (8,3) circle (3.0pt);
\fill (0,4) circle (4.2pt);
\fill (1,4) circle (4.2pt);
\fill[white] (2,4) circle (3.0pt); \draw (2,4) circle (3.0pt);
\fill (3,4) circle (4.2pt);
\fill (4,4) circle (4.2pt);
\fill[white] (5,4) circle (3.0pt); \draw (5,4) circle (3.0pt);
\fill (6,4) circle (4.2pt);
\fill (7,4) circle (4.2pt);
\fill[white] (8,4) circle (3.0pt); \draw (8,4) circle (3.0pt);
\fill (0,5) circle (4.2pt);
\fill (1,5) circle (4.2pt);
\fill[white] (2,5) circle (3.0pt); \draw (2,5) circle (3.0pt);
\fill (3,5) circle (4.2pt);
\fill (4,5) circle (4.2pt);
\fill[white] (5,5) circle (3.0pt); \draw (5,5) circle (3.0pt);
\fill (6,5) circle (4.2pt);
\fill (7,5) circle (4.2pt);
\fill[white] (8,5) circle (3.0pt); \draw (8,5) circle (3.0pt);
\fill[white] (0,6) circle (3.0pt); \draw (0,6) circle (3.0pt);
\fill[white] (1,6) circle (3.0pt); \draw (1,6) circle (3.0pt);
\fill[white] (2,6) circle (3.0pt); \draw (2,6) circle (3.0pt);
\fill[white] (3,6) circle (3.0pt); \draw (3,6) circle (3.0pt);
\fill[white] (4,6) circle (3.0pt); \draw (4,6) circle (3.0pt);
\fill[white] (5,6) circle (3.0pt); \draw (5,6) circle (3.0pt);
\fill[white] (6,6) circle (3.0pt); \draw (6,6) circle (3.0pt);
\fill[white] (7,6) circle (3.0pt); \draw (7,6) circle (3.0pt);
\fill[white] (8,6) circle (3.0pt); \draw (8,6) circle (3.0pt);
\fill[white] (0,7) circle (3.0pt); \draw (0,7) circle (3.0pt);
\fill[white] (1,7) circle (3.0pt); \draw (1,7) circle (3.0pt);
\fill[white] (2,7) circle (3.0pt); \draw (2,7) circle (3.0pt);
\fill[white] (3,7) circle (3.0pt); \draw (3,7) circle (3.0pt);
\fill[white] (4,7) circle (3.0pt); \draw (4,7) circle (3.0pt);
\fill[white] (5,7) circle (3.0pt); \draw (5,7) circle (3.0pt);
\fill[white] (6,7) circle (3.0pt); \draw (6,7) circle (3.0pt);
\fill[white] (7,7) circle (3.0pt); \draw (7,7) circle (3.0pt);
\fill[white] (8,7) circle (3.0pt); \draw (8,7) circle (3.0pt);
\draw[red, very thick, dashed] (2.55,1.55) rectangle (5.45,5.45);
\node at (4.00,-0.9) {(b)};
\end{scope}
\end{tikzpicture}
\caption{Two periodic idcodes of density $1/3$ of the elongated triangular grid (black vertices), drawn as the square grid in which the squares of alternate rows are cut by a diagonal at $45^\circ$; these rows are the rows of triangles.}
\label{fig:elong}
\end{figure}
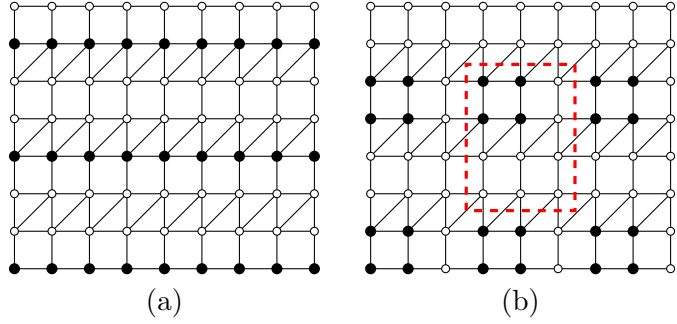

\begin{theorem}\label{thm:elong}
The minimum density of an idcode of the elongated triangular grid $3^3.4^2$ is $1/3$.
\end{theorem}

\begin{proof}
It is easy to check that the two sets of Figure~\ref{fig:elong} are periodic idcodes of density $1/3$.

Conversely, let $C$ be an idcode. Every codeword starts with charge $3$ and all other vertices with charge $0$. In the discharging phase every vertex $v\notin C$ receives exactly $1$ in total from the codewords of $I(v)$, and every codeword receives $1/2$ or nothing, according to the following rules.
\begin{itemize}
\item[(R1)] If $v\in C$ and $\beta(v)\notin C$, then $v$ sends $1$ to $\beta(v)$. Moreover, $v$ receives $1/2$ from a codeword $x$ in the following two situations: $x$ is the unique codeword neighbor of $v$, or $x$ is the unique inner neighbor of $v$ in $C$ and the outer neighbor of $v$ adjacent to $x$ is in $C$. In all other cases, a codeword receives nothing.
\item[(R2)] If $v\notin C$, $\beta(v)\notin C$ and $|I(v)|\ge 3$, then $v$ receives equal amounts (summing to $1$) from the inner neighbors of $v$ in $C$ that are adjacent to an outer neighbor of $v$ in $C$, and nothing from the other codewords. Such inner neighbors exist, since $I(v)$ contains three of the four outer and inner neighbors of $v$, and any three of them contain an outer neighbor together with its adjacent inner neighbor.
\item[(R3)] If $v\notin C$, $\beta(v)\notin C$ and $|I(v)|\le 2$, then $v$ receives $1/|I(v)|$ from each codeword of $I(v)$.
\end{itemize}
Every vertex outside $C$ ends with charge $1$, so it suffices to show that every codeword also ends with charge at least $1$. Let $u$ be a codeword, let $q$ be the total amount given by $u$ to its neighbors, and let $r\in\{0,\frac12\}$ be the amount that $u$ receives. Then $u$ ends with charge $3-q+r$, so we must show that $q\le 2+r$.

Let $b=\beta(u)$, and let $u_1,u_2,u_3,u_4$ be the other neighbors of $u$, in the cyclic order around $u$ starting after $b$. Thus the outer neighbors of $u$ are $u_1,u_4$ and its inner neighbors are $u_2,u_3$, and $u_1u_2u_3u_4$ is a path (see Figure~\ref{fig:elong-nbhd}). Let $P=\{u_1,u_2,u_3,u_4\}$. Also, the inner neighbors of $u_2$ are $u$ and $u_1$, and its outer neighbors are $u_3$ and a vertex $z_1$ adjacent to $u_1$; the vertex $z_1$ is the inner neighbor of $u_1$ other than $u_2$, and the outer neighbors of $u_1$ are $u$ and a vertex $u_0$. Symmetrically we define $z_4$ and $u_5$, so that $u_0,u_1,u,u_4,u_5$ are consecutive vertices of the horizontal line of $u$.

For $w\in N(u)$ we write $q_w$ for the amount that $u$ sends to $w$, so that $q=\sum_{w\in N(u)}q_w$. Every amount sent in the rules is $1/2$ or $1$, since in (R2) the charge is split among at most two inner neighbors and in (R3) we have $|I(v)|\le 2$. Hence $q_w\in\{0,1/2,1\}$ for every $w\in N(u)$. By (R1), we have $q_b=1$ and $r\le 1/2$ if $b\notin C$, and $q_b=r=0$ if $b\in C$.

\begin{figure}[ht]
\centering
\begin{tikzpicture}[scale=1.1]
\draw (-2.35,0)--(2.35,0);
\draw (-2.35,1)--(2.35,1);
\draw (-2.35,2)--(2.35,2);
\draw (-2,-0.35)--(-2,2.35);
\draw (-1,-0.35)--(-1,2.35);
\draw (0,-0.35)--(0,2.35);
\draw (1,-0.35)--(1,2.35);
\draw (2,-0.35)--(2,2.35);
\draw (-2,1)--(-1,2);
\draw (-1,1)--(0,2);
\draw (0,1)--(1,2);
\draw (1,1)--(2,2);
\draw (-2,2)--(-2.35,1.65);
\draw (2,1)--(2.35,1.35);
\fill[white] (-2,0) circle (2.2pt); \draw (-2,0) circle (2.2pt);
\fill[white] (-2,1) circle (2.2pt); \draw (-2,1) circle (2.2pt);
\node[below right,inner sep=2pt] at (-2,1) {$u_0$};
\fill[white] (-2,2) circle (2.2pt); \draw (-2,2) circle (2.2pt);
\fill[white] (-1,0) circle (2.2pt); \draw (-1,0) circle (2.2pt);
\fill[white] (-1,1) circle (2.2pt); \draw (-1,1) circle (2.2pt);
\node[below right,inner sep=2pt] at (-1,1) {$u_1$};
\fill[white] (-1,2) circle (2.2pt); \draw (-1,2) circle (2.2pt);
\node[above right,inner sep=2pt] at (-1,2) {$z_1$};
\fill[white] (0,0) circle (2.2pt); \draw (0,0) circle (2.2pt);
\node[below right,inner sep=2pt] at (0,0) {$b$};
\fill (0,1) circle (2.6pt);
\node[below right,inner sep=2pt] at (0,1) {$u$};
\fill[white] (0,2) circle (2.2pt); \draw (0,2) circle (2.2pt);
\node[above right,inner sep=2pt] at (0,2) {$u_2$};
\fill[white] (1,0) circle (2.2pt); \draw (1,0) circle (2.2pt);
\fill[white] (1,1) circle (2.2pt); \draw (1,1) circle (2.2pt);
\node[below right,inner sep=2pt] at (1,1) {$u_4$};
\fill[white] (1,2) circle (2.2pt); \draw (1,2) circle (2.2pt);
\node[above right,inner sep=2pt] at (1,2) {$u_3$};
\fill[white] (2,0) circle (2.2pt); \draw (2,0) circle (2.2pt);
\fill[white] (2,1) circle (2.2pt); \draw (2,1) circle (2.2pt);
\node[below right,inner sep=2pt] at (2,1) {$u_5$};
\fill[white] (2,2) circle (2.2pt); \draw (2,2) circle (2.2pt);
\node[above right,inner sep=2pt] at (2,2) {$z_4$};
\end{tikzpicture}
\caption{The names of the vertices around a codeword $u$ (black) in the proof of Theorem~\ref{thm:elong}, where $b=\beta(u)$. The outer neighbors are $u_1,u_4$ and the inner neighbors are $u_2,u_3$.}
\label{fig:elong-nbhd}
\end{figure}
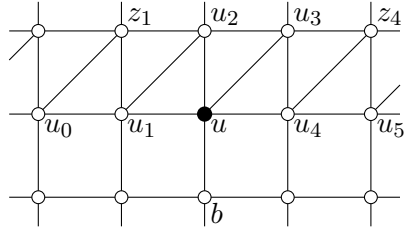

By the rules, the values $q_{u_1}$ and $q_{u_2}$ are described as follows. Symmetrically, the values $q_{u_4}$ and $q_{u_3}$ are obtained by exchanging $u_i$ with $u_{5-i}$, and $z_1$ with $z_4$.
\begin{itemize}
\item[$\bullet$] $q_{u_1}=1$ if and only if $I(u_1)=\{u\}$.
\item[$\bullet$] $q_{u_2}=1$ if and only if $I(u_2)=\{u\}$, or $u_2,\beta(u_2)\notin C$, $u_3\in C$ and $I(u_2)$ is $\{u,u_3,u_1\}$ or $\{u,u_3,z_1\}$.
\item[$\bullet$] $q_{u_1}>0$ if and only if $\beta(u_1)\notin C$ and either $u_1\in C$ and $I(u_1)=\{u_1,u\}$, or $u_1\notin C$ and $|I(u_1)|\le 2$.
\item[$\bullet$] $q_{u_2}>0$ if and only if $\beta(u_2)\notin C$ and either $u_2\in C$ and ($I(u_2)=\{u_2,u\}$, or $u_1\notin C$ and $u_3\in C$), or $u_2\notin C$ and ($|I(u_2)|\le 2$ or $u_3\in C$).
\end{itemize}
Indeed, if $w\in P$ is a codeword, then by (R1) we have $q_w\le 1/2$, and $q_w>0$ exactly when $\beta(w)\notin C$ and either $u$ is the unique codeword neighbor of $w$, or $w=u_2$, $u_1\notin C$ and $u_3\in C$; this is because $u$ is an outer neighbor of $u_1$, while $u$ is an inner neighbor of $u_2$ and $u_3$ is the outer neighbor of $u_2$ adjacent to $u$. If $w\notin C$ and $\beta(w)\in C$, then $q_w=0$. If $w\notin C$, $\beta(w)\notin C$ and $|I(w)|\le 2$, then $q_w=1/|I(w)|$ by (R3). Finally, suppose that $w\notin C$, $\beta(w)\notin C$ and $|I(w)|\ge 3$, so that (R2) applies. If $w=u_1$, then $q_w=0$, since $u$ is an outer neighbor of $u_1$. If $w=u_2$, then $I(u_2)\subseteq\{u,u_1,u_3,z_1\}$, so $q_{u_2}=0$ if $u_3\notin C$, while if $u_3\in C$ the vertex $u_2$ receives from $u$, and also from $u_1$ exactly when $u_1,z_1\in C$; thus $q_{u_2}=1/2$ if $I(u_2)=\{u,u_1,u_3,z_1\}$, and $q_{u_2}=1$ otherwise.

We use two consequences of this description. First, $q_w=1$ for at most one vertex $w\in P$. Indeed, at most one vertex $v$ of $N[u]$ satisfies $I(v)=\{u\}$, since $C$ is identifying. So if $q_w=1$ for two vertices $w\in P$, one of them would be, by symmetry, $u_2$ with $u_2\notin C$, $u_3\in C$ and $I(u_2)\in\{\{u,u_3,u_1\},\{u,u_3,z_1\}\}$. The other one is neither the codeword $u_3$ nor $u_4$, whose set $I(u_4)$ contains $u_3$, so it is $u_1$ with $I(u_1)=\{u\}$. But then $u_1\notin C$ and $z_1\notin C$, as $z_1$ is adjacent to $u_1$, which contradicts the form of $I(u_2)$. Second, if $u_1,u_2\in C$, then $q_{u_1}=q_{u_2}=0$, since $u_2\in I(u_1)$ and $u_1\in I(u_2)$.

We now distinguish cases according to $b$ and to the set $C\cap P$, up to the symmetry exchanging $u_i$ with $u_{5-i}$. We must show that $\sum_{w\in P}q_w\le 2+r-q_b$.

Suppose first that $b\in C$, so that $q_b=r=0$, and we must show that $\sum_{w\in P}q_w\le 2$. Since $q_w=1$ for at most one $w\in P$, this holds unless $q_w>0$ for all four vertices $w\in P$ and $q_w=1$ for one of them, and we show that this is impossible. If $u_1,u_2\in C$ or $u_3,u_4\in C$, this follows from the second consequence above. If $C\cap P$ is $\{u_1,u_4\}$ or $\{u_2,u_3\}$, then no $q_w$ equals $1$: every $w\in P$ is a codeword or has a codeword neighbor in $P$, the exceptional case for $q_{u_2}=1$ requires $u_2\notin C$ and $u_3\in C$, and the one for $q_{u_3}=1$ requires $u_3\notin C$ and $u_2\in C$. Up to symmetry, the remaining sets are the following.
\begin{itemize}
\item[$\bullet$] $C\cap P=\emptyset$. Since $u_3\notin C$, $q_{u_2}>0$ gives $\beta(u_2)\notin C$ and $|I(u_2)|\le 2$, so $I(u_2)$ is $\{u\}$ or $\{u,z_1\}$. The latter would give $z_1\in I(u_1)$, and $q_{u_1}>0$ would give $I(u_1)=\{u,z_1\}=I(u_2)$. Hence $I(u_2)=\{u\}$, and symmetrically $I(u_3)=\{u\}$, which is impossible.
\item[$\bullet$] $C\cap P=\{u_1\}$. Then $q_{u_1}>0$ gives $I(u_1)=\{u_1,u\}$. Since $u_3\notin C$, $q_{u_2}>0$ gives $\beta(u_2)\notin C$ and $|I(u_2)|\le 2$, so $I(u_2)=\{u,u_1\}=I(u_1)$, which is impossible.
\item[$\bullet$] $C\cap P=\{u_2\}$. Then $q_{u_1}>0$ gives $|I(u_1)|\le 2$, so $I(u_1)=\{u,u_2\}$. Since $u_3\notin C$, $q_{u_2}>0$ gives $I(u_2)=\{u_2,u\}=I(u_1)$, which is impossible.
\item[$\bullet$] $C\cap P=\{u_1,u_3\}$. Since $u_2\notin C$, $q_{u_3}>0$ gives $I(u_3)=\{u_3,u\}$, and $q_{u_4}>0$ gives $|I(u_4)|\le 2$, so $I(u_4)=\{u,u_3\}=I(u_3)$, which is impossible.
\end{itemize}

Suppose now that $b\notin C$, so that $q_b=1$ and $I(u)=\{u\}\cup(C\cap P)$, and we must show that $\sum_{w\in P}q_w\le 1+r$. If $u_1,u_2\in C$, then $q_{u_1}=q_{u_2}=0$. If moreover $u_3\notin C$, then $r=1/2$ by (R1), and $q_{u_3}+q_{u_4}\le 3/2$, since at most one of them equals $1$. If $u_3\in C$, then $q_{u_3}\le 1/2$, and $q_{u_4}\le 1/2$ since $u_3\in I(u_4)$, so the sum is at most $1$. Up to symmetry, the remaining sets are the following.
\begin{itemize}
\item[$\bullet$] $C\cap P=\emptyset$. Then $r=0$ and $I(u)=\{u\}$, so no $w\in P$ satisfies $I(w)=\{u\}$, and $q_w\le 1/2$ for all $w\in P$. As in the case $b\in C$, if $q_{u_1}>0$ and $q_{u_2}>0$, then $I(u_2)=\{u\}$, which is impossible. Hence $q_w>0$ for at most one of $u_1,u_2$ and at most one of $u_3,u_4$, and the sum is at most $1$.
\item[$\bullet$] $C\cap P=\{u_1\}$. Then $r=1/2$ by (R1), and $I(u)=\{u,u_1\}$. We have $q_{u_1}=0$, since otherwise $I(u_1)=\{u_1,u\}=I(u)$, and $q_{u_2}=0$, since otherwise $\beta(u_2)\notin C$ and $|I(u_2)|\le 2$, which gives $I(u_2)=\{u,u_1\}=I(u)$. As at most one of $q_{u_3},q_{u_4}$ equals $1$, the sum is at most $3/2$.
\item[$\bullet$] $C\cap P=\{u_2\}$. Then $r=1/2$ by (R1), and $I(u)=\{u,u_2\}$. We have $q_{u_2}=0$, since otherwise $I(u_2)=\{u_2,u\}=I(u)$, and $q_{u_1}=0$, since otherwise $I(u_1)=\{u,u_2\}=I(u)$. As before, the sum is at most $3/2$.
\item[$\bullet$] $C\cap P=\{u_1,u_3\}$. Then $r=0$ and $I(u)=\{u,u_1,u_3\}$. Every $q_w$ is at most $1/2$: this is clear for the codewords $u_1,u_3$ and for $u_4$, whose set $I(u_4)$ contains $u_3$, and $q_{u_2}=1$ would give $I(u_2)=\{u,u_3,u_1\}=I(u)$. If $q_{u_1}>0$, then $I(u_1)=\{u_1,u\}$, so $z_1\notin C$; if moreover $q_{u_2}>0$, then $\beta(u_2)\notin C$ and $I(u_2)=\{u,u_1,u_3\}=I(u)$, which is impossible. If $q_{u_3}>0$ and $q_{u_4}>0$, then $I(u_3)=\{u_3,u\}=I(u_4)$, as in the case $b\in C$. Hence the sum is at most $1$.
\item[$\bullet$] $C\cap P=\{u_1,u_4\}$. Then $r=0$, and every $q_w$ is at most $1/2$. If $q_{u_1}>0$ and $q_{u_2}>0$, then $I(u_1)=\{u_1,u\}$, and, since $u_3\notin C$, $I(u_2)=\{u,u_1\}=I(u_1)$, which is impossible. By symmetry, the sum is at most $1$.
\item[$\bullet$] $C\cap P=\{u_2,u_3\}$. Then $r=0$ and $I(u)=\{u,u_2,u_3\}$, and every $q_w$ is at most $1/2$. If $q_{u_1}>0$ and $q_{u_2}>0$, then $I(u_1)=\{u,u_2\}$, so $z_1\notin C$, and $\beta(u_2)\notin C$, so $I(u_2)=\{u_2,u,u_3\}=I(u)$, which is impossible. By symmetry, the sum is at most $1$.
\end{itemize}
In all cases $q\le 2+r$.

Hence every vertex ends with charge at least $1$. Every codeword started with charge $3$ and charge only moves between adjacent vertices, so Theorem~\ref{thm:discharging} gives $\dens(C)\ge 1/3$.
\end{proof}

\section{The snub hexagonal grid}\label{sec:snub}

The snub hexagonal tiling is the only chiral Archimedean tiling, that is, the only one that cannot be mapped onto its mirror image by an isometry; all its vertices have configuration $3^4.6$ (see Figure~\ref{fig:snub}). Its graph is $5$-regular, so the general lower bound of~\cite{SSW24} gives $2/7$. Its minimum density was previously known to lie in the interval $(3/10,1/3]$~\cite{Alb11}. In this section we prove that it equals $1/3$, which is item~(e) of Theorem~\ref{thm:main}.

Every vertex of the snub hexagonal grid lies on exactly one hexagon, and its four other faces are triangles. Two hexagons are \emph{adjacent} if an edge joins them. Every hexagon is adjacent to six others, and with this adjacency the hexagons form a triangular grid (see Figure~\ref{fig:snub}).

For a vertex $v$, we list the five neighbors of $v$ counterclockwise, starting just after the hexagon of $v$, as $v_1,v_2,v_3,v_4,v_5$ (Figure~\ref{fig:snub-ball}), and we call $v_i$ the \emph{$i$-th neighbor} of $v$. The following facts can be checked in Figures~\ref{fig:snub-ball} and~\ref{fig:snub}. Since the symmetries of the tiling act transitively on the vertices and preserve the orientation, it suffices to check them for one vertex.
\begin{itemize}
\item[(F1)] The edges among the neighbors of $v$ are exactly $v_1v_2$, $v_2v_3$, $v_3v_4$ and $v_4v_5$.
\item[(F2)] The vertices $v_1,v_5$ lie on the hexagon of $v$, the vertices $v_3,v_4$ lie on a second hexagon, and $v_2$ lies on a third one; these three hexagons are pairwise adjacent.
\item[(F3)] The vertex $v$ is the fifth, third, second, fourth and first neighbor of $v_1,v_2,v_3,v_4,v_5$, respectively.
\end{itemize}

\begin{figure}[t]
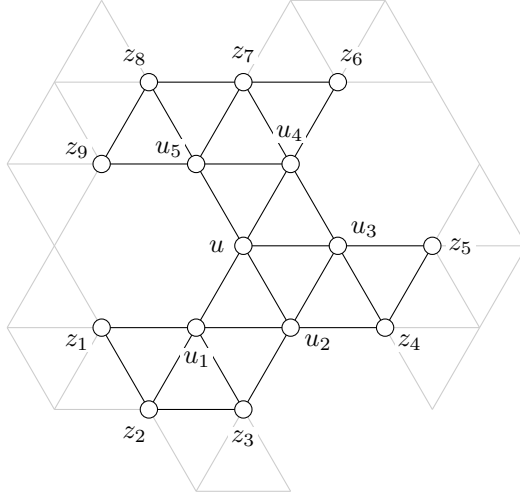

\centering

\caption{The ball $N_2[u]$ of radius two around a vertex $u$ of the snub hexagonal grid, with the names used in the proof of Theorem~\ref{thm:snub}; edges leaving $N_2[u]$ are drawn in gray. The hexagon of $u$ contains $u_1,u_5,z_1,z_9$, the hexagon of $u_3,u_4$ contains $z_5,z_6$, and the hexagon of $u_2$ contains $z_3,z_4$.}
\label{fig:snub-ball}
\end{figure}

\begin{figure}[t]
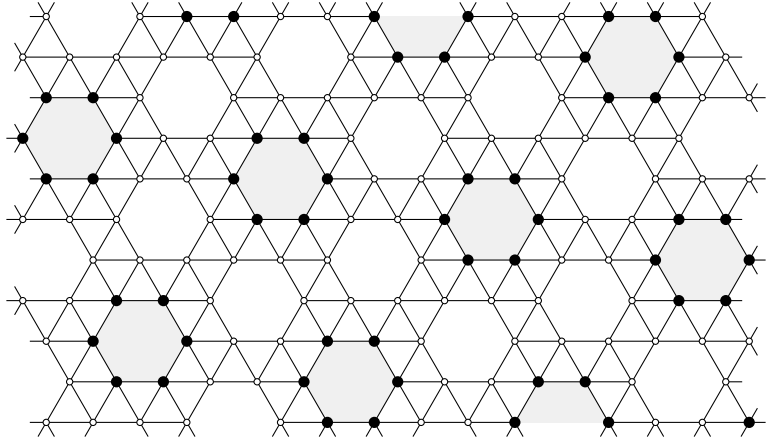

\centering
%
\caption{The idcode of density $1/3$ of the snub hexagonal grid in the proof of Theorem~\ref{thm:snub} (black vertices), consisting of one color class of hexagons: any two closest hexagons of the idcode have two common adjacent hexagons.}
\label{fig:snub}
\end{figure}

\begin{theorem}\label{thm:snub}
The minimum density of an idcode of the snub hexagonal grid $3^4.6$ is $1/3$.
\end{theorem}

\begin{proof}
The upper bound is given by a code made of whole hexagons. The triangular grid formed by the hexagons has a unique proper $3$-coloring, up to permutation of the colors, since the colors of one triangle determine all the others. This coloring is periodic, and each color class contains one third of the hexagons. Let $C$ be the union of the hexagons of one color class, as shown in Figure~\ref{fig:snub}. Since the hexagons partition the vertex set and have six vertices each, the density of $C$ is $1/3$, and it is easy to check with Figure~\ref{fig:snub} that $C$ is an idcode.

Conversely, let $C$ be an idcode. We use a scheme similar to the one in Section~\ref{sec:elong}: every codeword starts with charge $3$ and all other vertices with charge $0$, and in the discharging phase every vertex $v\notin C$ receives exactly $1$ in total from the codewords of $I(v)$, and every codeword receives $1/2$ or nothing. For a vertex $v$, let $I'(v)=I(v)\setminus\{v\}=N(v)\cap C$ be the set of codeword neighbors of $v$. The amounts received by a vertex $v$ depend only on $I(v)$ and are given by the following rules, in which the codewords of $N(v)$ are designated by their names $v_1,\dots,v_5$ relative to $v$.
\begin{itemize}
\item[(R1)] If $v\notin C$, then $v$ receives $1/|I(v)|$ from each codeword of $I(v)$, except when $I'(v)$ is one of the seventeen sets of Table~\ref{tab:snub-rules}(a); in that case $v$ receives equal amounts (summing to $1$) from the codewords listed in the second column, and nothing from the other codewords of $I(v)$.
\item[(R2)] If $v\in C$ and $I'(v)$ is one of the seven sets of Table~\ref{tab:snub-rules}(b), then $v$ receives $1/2$ from the codeword listed in the second column. Otherwise a codeword receives nothing.
\end{itemize}

\begin{table}[t]
\centering
\small
\begin{tabular}[t]{ll}
\multicolumn{2}{c}{(a) $v\notin C$}\\
\hline
\rowcolor{gray!25}$I'(v)$ & receives from\\
\hline
$\{v_2,v_5\}$ & $v_5$ \\
$\{v_2,v_3\}$ & $v_2$ \\
$\{v_1,v_3\}$ & $v_1$ \\
$\{v_1,v_2,v_5\}$ & $v_1,v_5$ \\
$\{v_1,v_2,v_3\}$ & $v_1,v_2$ \\
$\{v_1,v_2,v_4\}$ & $v_1,v_2$ \\
$\{v_1,v_3,v_4\}$ & $v_1,v_3$ \\
$\{v_1,v_3,v_5\}$ & $v_1,v_5$ \\
$\{v_2,v_3,v_4\}$ & $v_2,v_3$ \\
$\{v_2,v_3,v_5\}$ & $v_2,v_5$ \\
$\{v_2,v_4,v_5\}$ & $v_4,v_5$ \\
$\{v_3,v_4,v_5\}$ & $v_3,v_4$ \\
$\{v_1,v_2,v_3,v_4\}$ & $v_2,v_3$ \\
$\{v_1,v_2,v_3,v_5\}$ & $v_2,v_5$ \\
$\{v_1,v_3,v_4,v_5\}$ & $v_3,v_4$ \\
$\{v_2,v_3,v_4,v_5\}$ & $v_3,v_4$ \\
$\{v_1,v_2,v_3,v_4,v_5\}$ & $v_2,v_3,v_4$ \\
\hline
\end{tabular}
\hspace{4em}
\begin{tabular}[t]{lc}
\multicolumn{2}{c}{(b) $v\in C$}\\
\hline
\rowcolor{gray!25}$I'(v)$ & receives $1/2$ from\\
\hline
$\{v_1\}$ & $v_1$ \\
$\{v_3\}$ & $v_3$ \\
$\{v_4\}$ & $v_4$ \\
$\{v_1,v_2\}$ & $v_2$ \\
$\{v_2,v_4\}$ & $v_4$ \\
$\{v_4,v_5\}$ & $v_5$ \\
$\{v_1,v_2,v_5\}$ & $v_2$ \\
\hline
\end{tabular}
\caption{The exceptional cases of rules (R1) and (R2), where $v_1,\ldots,v_5$ are the neighbors of $v$ and $I'(v)=I(v)\setminus\{v\}$.}
\label{tab:snub-rules}
\end{table}

Every vertex outside $C$ ends with charge $1$, so it suffices to show that every codeword also ends with charge at least $1$. Let $u$ be a codeword, let $q$ be the total amount given by $u$ to its neighbors, and let $r\in\{0,\frac12\}$ be the amount that $u$ receives. Then $u$ ends with charge $3-q+r$, so we must show that $q\le 2+r$.

For $w\in N(u)$, let $q_w$ be the amount that $u$ sends to $w$, so that $q=\sum_{w\in N(u)}q_w$. We name the vertices of $N_2[u]$ as in Figure~\ref{fig:snub-ball}: $u_1,\dots,u_5$ are the neighbors of $u$, and $z_1,\dots,z_9$ are the vertices at distance two. By (F3), $u$ is the fifth, third, second, fourth and first neighbor of $u_1,u_2,u_3,u_4,u_5$, respectively, and this is how $u$ appears in the rules applied by these vertices.

We first record what the rules give. Clearly $r\in\{0,\frac12\}$. Let $w\in N(u)$. If $w\in C$, then $q_w$ is $\frac12$ in the seven situations of Table~\ref{tab:snub-rules}(b) in which the listed codeword is $u$, and $0$ otherwise. If $w\notin C$, then $q_w=1$ when $I(w)=\{u\}$, and also in the three exceptional cases with $|I(w)|=2$ of Table~\ref{tab:snub-rules}(a) in which $u$ is the listed codeword; in all other cases $q_w\le\frac12$. Thus $q_w\le 1$ always, and $q_w\in\{0,\frac12,1\}$ unless $|I(w)|\ge 3$.

The value $q_w$ depends only on $N[w]\cap C$, and $N[w]\subseteq N_2[u]$. Table~\ref{tab:snub-cert} is a case analysis. First, the cases are separated according to the set $I'(u)$ of codeword neighbors of $u$. Then, for some sets $I'(u)$, each case is subdivided according to whether some further vertices of $N_2[u]$ are in $C$. In each case, these further assumptions, together with $I'(u)$ and $u\in C$, determine whether some vertices of $N_2[u]$ are in $C$. Each row of the table gives:
\begin{itemize}
\item[$\bullet$] a lower bound for $r$;
\item[$\bullet$] for each $w\in N(u)$, the maximum possible value of $q_w$, over all ways of deciding which vertices of $N[w]$ not determined by the row are in $C$, subject only to the following two conditions, which every idcode satisfies: $I(x)\neq\emptyset$ for every $x$ with $N[x]\subseteq N_2[u]$, and $I(x)\neq I(y)$, that is, $N[x]\triangle N[y]$ contains a codeword, for all $x\in N_2[u]$ and $y\neq x$ with $N[x]\triangle N[y]\subseteq N_2[u]$.
\end{itemize}
These maxima are obtained by reading Table~\ref{tab:snub-rules} together with the observations of the previous paragraph. For instance, in the first row with $I'(u)=\emptyset$ that is not impossible, the vertex $u_2$ receives nothing from $u$: its codeword neighbors are $u$ and $z_4$, which are its third and first neighbors, so $I(u_2)=\{v_1,v_3\}$ in the names relative to $v=u_2$, and by Table~\ref{tab:snub-rules}(a) the vertex $u_2$ receives everything from $z_4$. The rows marked impossible are those in which the fixed vertices already violate one of the two conditions: for instance, if $I'(u)=\emptyset$ and $z_3,z_4\notin C$, then $I(u_2)=\{u\}=I(u)$. In all the other rows the upper bounds add up to at most $2$ plus the lower bound for $r$, and the rows cover all possibilities for $I'(u)$ and the further vertices. Hence $q\le 2+r$ in all cases, and every vertex ends with charge at least $1$. Every codeword started with charge $3$ and charge only moves between adjacent vertices, so Theorem~\ref{thm:discharging} gives $\dens(C)\ge 1/3$.
\end{proof}

\afterpage{{\small
\setlength{\tabcolsep}{4pt}
\begin{longtable}{l|l|l|>{\columncolor{gray!10}}c>{\columncolor{gray!10}}c>{\columncolor{gray!10}}c>{\columncolor{gray!10}}c>{\columncolor{gray!10}}c>{\columncolor{gray!10}}c>{\columncolor{gray!10}}c}
\caption{Certificate for the proof of Theorem~\ref{thm:snub}. The first column is $I'(u)=I(u)\setminus\{u\}$, the next two columns list further assumptions on $N_2[u]$, namely the vertices assumed to be in $C$ and those assumed not to be in $C$, the next five columns give upper bounds for $q_{u_1},\dots,q_{u_5}$, the next one their sum, and the last one a lower bound for $r$. In every possible row, the sum is at most $2+r$.}\label{tab:snub-cert}\\
\hline
\rowcolor{gray!25} & \multicolumn{2}{c|}{further assumptions} & & & & & & & \\
\rowcolor{gray!25}$I'(u)$ & \multicolumn{1}{c|}{$\in C$} & \multicolumn{1}{c|}{$\notin C$} & $u_1$ & $u_2$ & $u_3$ & $u_4$ & $u_5$ & sum & $r$\\
\hline
\endfirsthead
\hline
\rowcolor{gray!25} & \multicolumn{2}{c|}{further assumptions} & & & & & & & \\
\rowcolor{gray!25}$I'(u)$ & \multicolumn{1}{c|}{$\in C$} & \multicolumn{1}{c|}{$\notin C$} & $u_1$ & $u_2$ & $u_3$ & $u_4$ & $u_5$ & sum & $r$\\
\hline
\endhead
\hline
\endfoot
$\emptyset$ &  & $z_3,z_4$ & \multicolumn{7}{>{\columncolor{gray!10}}l}{impossible: $I(u)=I(u_2)$} \\
 & $z_4$ & $z_3,z_5$ & \multicolumn{7}{>{\columncolor{gray!10}}l}{impossible: $I(u_2)=I(u_3)$} \\
 & $z_4,z_5$ & $z_3,z_8,z_9$ & $1$ & $0$ & $0$ & $0$ & $1$ & $2$ & $0$ \\
 & $z_4,z_5,z_8$ & $z_3,z_9$ & $1$ & $0$ & $0$ & $\frac{1}{2}$ & $\frac{1}{2}$ & $2$ & $0$ \\
 & $z_4,z_5,z_9$ & $z_3$ & $1$ & $0$ & $0$ & $\frac{1}{2}$ & $\frac{1}{2}$ & $2$ & $0$ \\
 & $z_3$ & $z_4,z_5$ & \multicolumn{7}{>{\columncolor{gray!10}}l}{impossible: $I(u)=I(u_3)$} \\
 & $z_3,z_5$ & $z_4,z_8,z_9$ & $\frac{1}{2}$ & $\frac{1}{2}$ & $0$ & $0$ & $1$ & $2$ & $0$ \\
 & $z_3,z_5,z_8$ & $z_4,z_9$ & $\frac{1}{2}$ & $\frac{1}{2}$ & $0$ & $\frac{1}{2}$ & $\frac{1}{2}$ & $2$ & $0$ \\
 & $z_3,z_5,z_9$ & $z_4$ & $\frac{1}{2}$ & $\frac{1}{2}$ & $0$ & $\frac{1}{2}$ & $\frac{1}{2}$ & $2$ & $0$ \\
 & $z_3,z_4$ & $z_8,z_9$ & $\frac{1}{2}$ & $0$ & $\frac{1}{2}$ & $0$ & $1$ & $2$ & $0$ \\
 & $z_3,z_4,z_8$ & $z_9$ & $\frac{1}{2}$ & $0$ & $\frac{1}{2}$ & $\frac{1}{2}$ & $\frac{1}{2}$ & $2$ & $0$ \\
 & $z_3,z_4,z_9$ &  & $\frac{1}{2}$ & $0$ & $\frac{1}{2}$ & $\frac{1}{2}$ & $\frac{1}{2}$ & $2$ & $0$ \\
\hline
$u_5$ &  & $z_3$ & $1$ & $1$ & $0$ & $0$ & $0$ & $2$ & $0$ \\
 & $z_3$ &  & $\frac{1}{2}$ & $\frac{1}{2}$ & $1$ & $0$ & $0$ & $2$ & $0$ \\
\hline
$u_4$ &  & $z_4,z_6$ & $1$ & $1$ & $0$ & $\frac{1}{2}$ & $0$ & $\frac{5}{2}$ & $\frac{1}{2}$ \\
 & $z_6$ & $z_4$ & $1$ & $1$ & $0$ & $0$ & $\frac{1}{2}$ & $\frac{5}{2}$ & $\frac{1}{2}$ \\
 & $z_4$ &  & $1$ & $0$ & $\frac{1}{2}$ & $\frac{1}{2}$ & $\frac{1}{2}$ & $\frac{5}{2}$ & $\frac{1}{2}$ \\
\hline
$u_4,u_5$ & --- & --- & $1$ & $1$ & $\frac{1}{2}$ & $0$ & $0$ & $\frac{5}{2}$ & $\frac{1}{2}$ \\
\hline
$u_3$ & --- & --- & $1$ & $0$ & $0$ & $\frac{1}{2}$ & $1$ & $\frac{5}{2}$ & $\frac{1}{2}$ \\
\hline
$u_3,u_5$ & --- & --- & $1$ & $0$ & $0$ & $\frac{1}{2}$ & $\frac{1}{2}$ & $2$ & $0$ \\
\hline
$u_3,u_4$ & --- & --- & $1$ & $0$ & $\frac{1}{2}$ & $0$ & $\frac{1}{2}$ & $2$ & $0$ \\
\hline
$u_3,u_4,u_5$ & --- & --- & $1$ & $0$ & $\frac{1}{2}$ & $0$ & $0$ & $\frac{3}{2}$ & $0$ \\
\hline
$u_2$ &  & $z_3,z_6$ & $\frac{1}{2}$ & $0$ & $0$ & $1$ & $\frac{1}{2}$ & $2$ & $0$ \\
 & $z_3$ & $z_6$ & $0$ & $0$ & $\frac{1}{2}$ & $1$ & $\frac{1}{2}$ & $2$ & $0$ \\
 & $z_6$ & $z_3$ & $\frac{1}{2}$ & $0$ & $0$ & $\frac{1}{2}$ & $1$ & $2$ & $0$ \\
 & $z_3,z_6$ &  & $0$ & $0$ & $\frac{1}{2}$ & $\frac{1}{2}$ & $1$ & $2$ & $0$ \\
\hline
$u_2,u_5$ &  & $z_3,z_6$ & $\frac{1}{2}$ & $\frac{1}{2}$ & $\frac{1}{2}$ & $\frac{1}{2}$ & $0$ & $2$ & $0$ \\
 & $z_6$ & $z_3$ & $\frac{1}{2}$ & $\frac{1}{2}$ & $\frac{1}{2}$ & $0$ & $\frac{1}{2}$ & $2$ & $0$ \\
 & $z_3$ &  & $0$ & $0$ & $1$ & $\frac{1}{2}$ & $\frac{1}{2}$ & $2$ & $0$ \\
\hline
$u_2,u_4$ & --- & --- & $\frac{1}{2}$ & $\frac{1}{2}$ & $\frac{1}{2}$ & $\frac{1}{2}$ & $\frac{1}{2}$ & $\frac{5}{2}$ & $\frac{1}{2}$ \\
\hline
$u_2,u_4,u_5$ & --- & --- & $\frac{1}{2}$ & $\frac{1}{2}$ & $\frac{1}{2}$ & $0$ & $0$ & $\frac{3}{2}$ & $0$ \\
\hline
$u_2,u_3$ & --- & --- & $\frac{1}{2}$ & $0$ & $0$ & $\frac{1}{2}$ & $1$ & $2$ & $0$ \\
\hline
$u_2,u_3,u_5$ & --- & --- & $\frac{1}{2}$ & $0$ & $0$ & $\frac{1}{2}$ & $\frac{1}{2}$ & $\frac{3}{2}$ & $0$ \\
\hline
$u_2,u_3,u_4$ & --- & --- & $\frac{1}{2}$ & $0$ & $0$ & $0$ & $\frac{1}{2}$ & $1$ & $0$ \\
\hline
$u_2,u_3,u_4,u_5$ & --- & --- & $\frac{1}{2}$ & $0$ & $0$ & $0$ & $0$ & $\frac{1}{2}$ & $0$ \\
\hline
$u_1$ &  & $z_5,z_6,z_7$ & $0$ & $\frac{1}{2}$ & $\frac{1}{2}$ & $1$ & $\frac{1}{2}$ & $\frac{5}{2}$ & $\frac{1}{2}$ \\
 & $z_7$ & $z_5,z_6$ & $0$ & $\frac{1}{2}$ & $1$ & $\frac{1}{2}$ & $\frac{1}{2}$ & $\frac{5}{2}$ & $\frac{1}{2}$ \\
 & $z_6$ & $z_5,z_7,z_8,z_9$ & $0$ & $\frac{1}{2}$ & $\frac{1}{2}$ & $\frac{1}{2}$ & $1$ & $\frac{5}{2}$ & $\frac{1}{2}$ \\
 & $z_6,z_7$ & $z_5,z_8,z_9$ & $0$ & $\frac{1}{2}$ & $1$ & $0$ & $1$ & $\frac{5}{2}$ & $\frac{1}{2}$ \\
 & $z_6,z_8$ & $z_5,z_9$ & $0$ & $\frac{1}{2}$ & $1$ & $\frac{1}{2}$ & $\frac{1}{2}$ & $\frac{5}{2}$ & $\frac{1}{2}$ \\
 & $z_6,z_9$ & $z_5$ & $0$ & $\frac{1}{2}$ & $1$ & $\frac{1}{2}$ & $\frac{1}{2}$ & $\frac{5}{2}$ & $\frac{1}{2}$ \\
 & $z_5$ &  & $0$ & $\frac{1}{2}$ & $0$ & $1$ & $1$ & $\frac{5}{2}$ & $\frac{1}{2}$ \\
\hline
$u_1,u_5$ &  & $z_6$ & $0$ & $\frac{1}{2}$ & $1$ & $\frac{1}{2}$ & $0$ & $2$ & $0$ \\
 & $z_6$ &  & $0$ & $\frac{1}{2}$ & $1$ & $0$ & $\frac{1}{2}$ & $2$ & $0$ \\
\hline
$u_1,u_4$ & --- & --- & $0$ & $\frac{1}{2}$ & $\frac{1}{2}$ & $\frac{1}{2}$ & $\frac{1}{2}$ & $2$ & $0$ \\
\hline
$u_1,u_4,u_5$ & --- & --- & $0$ & $\frac{1}{2}$ & $\frac{1}{2}$ & $0$ & $0$ & $1$ & $0$ \\
\hline
$u_1,u_3$ & --- & --- & $0$ & $\frac{1}{2}$ & $0$ & $\frac{1}{2}$ & $1$ & $2$ & $0$ \\
\hline
$u_1,u_3,u_5$ & --- & --- & $0$ & $\frac{1}{2}$ & $0$ & $\frac{1}{2}$ & $\frac{1}{2}$ & $\frac{3}{2}$ & $0$ \\
\hline
$u_1,u_3,u_4$ & --- & --- & $0$ & $\frac{1}{2}$ & $\frac{1}{2}$ & $0$ & $\frac{1}{2}$ & $\frac{3}{2}$ & $0$ \\
\hline
$u_1,u_3,u_4,u_5$ & --- & --- & $0$ & $\frac{1}{2}$ & $\frac{1}{2}$ & $0$ & $0$ & $1$ & $0$ \\
\hline
$u_1,u_2$ &  & $z_6$ & $0$ & $0$ & $1$ & $1$ & $\frac{1}{2}$ & $\frac{5}{2}$ & $\frac{1}{2}$ \\
 & $z_6$ &  & $0$ & $0$ & $1$ & $\frac{1}{2}$ & $1$ & $\frac{5}{2}$ & $\frac{1}{2}$ \\
\hline
$u_1,u_2,u_5$ & --- & --- & $\frac{1}{2}$ & $0$ & $1$ & $\frac{1}{2}$ & $\frac{1}{2}$ & $\frac{5}{2}$ & $\frac{1}{2}$ \\
\hline
$u_1,u_2,u_4$ & --- & --- & $\frac{1}{2}$ & $0$ & $\frac{1}{2}$ & $\frac{1}{2}$ & $\frac{1}{2}$ & $2$ & $0$ \\
\hline
$u_1,u_2,u_4,u_5$ & --- & --- & $\frac{1}{2}$ & $0$ & $\frac{1}{2}$ & $0$ & $0$ & $1$ & $0$ \\
\hline
$u_1,u_2,u_3$ & --- & --- & $\frac{1}{2}$ & $0$ & $0$ & $\frac{1}{2}$ & $1$ & $2$ & $0$ \\
\hline
$u_1,u_2,u_3,u_5$ & --- & --- & $\frac{1}{2}$ & $0$ & $0$ & $\frac{1}{2}$ & $\frac{1}{2}$ & $\frac{3}{2}$ & $0$ \\
\hline
$u_1,u_2,u_3,u_4$ & --- & --- & $\frac{1}{2}$ & $0$ & $0$ & $0$ & $\frac{1}{2}$ & $1$ & $0$ \\
\hline
$u_1,u_2,u_3,u_4,u_5$ & --- & --- & $\frac{1}{2}$ & $0$ & $0$ & $0$ & $0$ & $\frac{1}{2}$ & $0$ \\
\end{longtable}
}}

In the code of Figure~\ref{fig:snub}, every codeword has exactly its two hexagon neighbors as codeword neighbors, and the sets $I(v)$ have sizes $1$, $2$ and $3$ in equal proportions; every codeword ends with charge exactly $1$. The rules (R1) and (R2) and the case analysis of Table~\ref{tab:snub-cert} were found and checked with the help of a computer~\cite{rudini26idc}, but the verification of each row only requires reading Table~\ref{tab:snub-rules}. Unlike the other grids treated here, the snub hexagonal grid does not seem to admit a short discharging argument: no rule in which every vertex simply splits its charge according to fixed priorities among the positions $v_1,\dots,v_5$ achieves the bound $1/3$.

\section{The remaining grids}\label{sec:remaining}

Together with the triangular and square grids, Theorem~\ref{thm:main} settles seven of the eleven Archimedean grids. Besides the hexagonal grid, the problem remains open for the rhombitrihexagonal grid $3.4.6.4$, the truncated trihexagonal grid $4.6.12$ and the snub square grid $3^2.4.3.4$. For these three grids we searched for periodic idcodes of minimum density with a constraint solver, over the period lattices with fundamental domains of up to about $100$ to $150$ vertices, with a time limit for each lattice. The best codes we found give the upper bounds of the following result, which is items~(f)--(h) of Theorem~\ref{thm:main}.

\begin{theorem}\label{thm:remaining}
The minimum density of an idcode lies
\begin{itemize}
\item[$\bullet$] between $3/7$ and $13/30$ for the truncated trihexagonal grid $4.6.12$;
\item[$\bullet$] between $1/3$ and $11/30$ for the rhombitrihexagonal grid $3.4.6.4$;
\item[$\bullet$] between $2/7$ and $9/28$ for the snub square grid $3^2.4.3.4$.
\end{itemize}
\end{theorem}

\begin{proof}
The truncated trihexagonal grid is twin-free and cubic, and every vertex lies on exactly one $4$-cycle, so Theorem~\ref{thm:4cycle} gives the lower bound $3/7$. The rhombitrihexagonal and snub square grids are $4$-regular and $5$-regular, respectively, so the general bound $2/(\Delta+2)$ of~\cite[Theorem~4.1]{SSW24} gives $1/3$ and $2/7$. The upper bounds are given by the periodic codes shown in Figures~\ref{fig:trtrihex}, \ref{fig:rhomb} and~\ref{fig:snubsq}, which have densities $13/30$, $11/30$ and $9/28$, respectively. Since two vertices at distance at least $3$ have disjoint closed neighborhoods, by periodicity it suffices to check that $I(u)\neq\emptyset$ and $I(u)\neq I(v)$ for every vertex $u$ of a fundamental domain of the period lattice and every vertex $v$ with $d(u,v)\le 2$. This is a routine verification with Figures~\ref{fig:rhomb}, \ref{fig:trtrihex} and~\ref{fig:snubsq}.
\end{proof}

\begin{figure}[t]
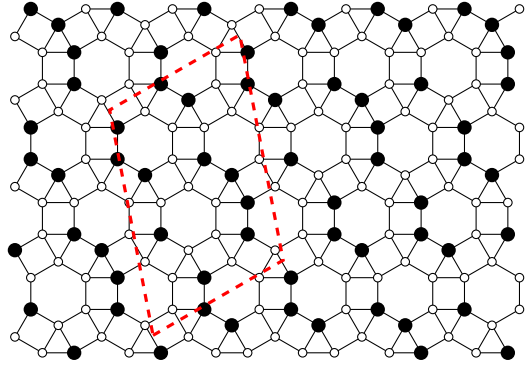

\centering

\caption{An idcode of density $11/30$ of the rhombitrihexagonal grid $3.4.6.4$ (black vertices). The code is periodic, and the dashed parallelogram is a fundamental domain of its period lattice: it contains $30$ vertices, $11$ of which are codewords.}
\label{fig:rhomb}
\end{figure}

\begin{figure}[t]
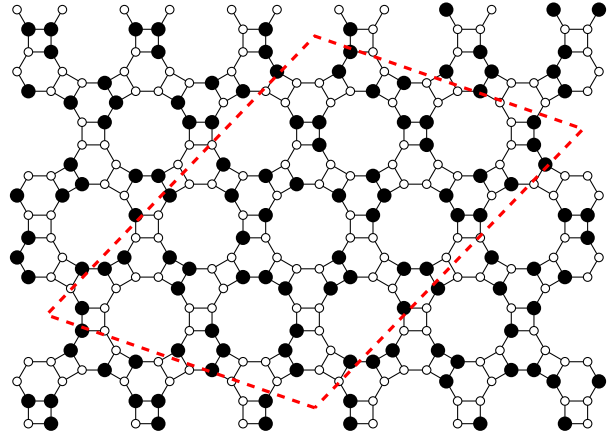

\centering
%
\caption{An idcode of density $13/30$ of the truncated trihexagonal grid $4.6.12$ (black vertices). The code is periodic, and the dashed parallelogram is a fundamental domain of its period lattice: it contains $120$ vertices, $52$ of which are codewords.}
\label{fig:trtrihex}
\end{figure}

\begin{figure}[t]
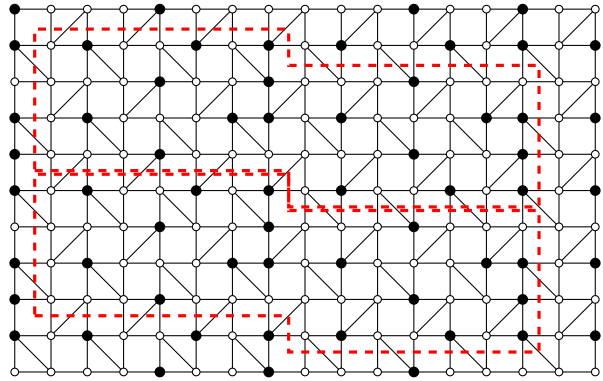

\centering
%
\caption{An idcode of density $9/28$ of the snub square grid $3^2.4.3.4$ (black vertices), drawn as the square grid in which alternate squares of each row are cut by a diagonal, at $45^\circ$ in one row and at $-45^\circ$ in the next, so that every vertex lies on exactly one diagonal. The code is periodic, and each of the two dashed regions is a fundamental domain of its period lattice: it contains $56$ vertices, $18$ of which are codewords.}
\label{fig:snubsq}
\end{figure}

For the truncated trihexagonal grid, the gap is small: $3/7\approx 0.4286$ and $13/30\approx 0.4333$. For the rhombitrihexagonal and snub square grids, the lower bounds of Theorem~\ref{thm:remaining} are only the general bounds of~\cite{SSW24}. Note that the snub square grid, which is $5$-regular and whose faces are triangles and squares, has idcodes of density smaller than $1/3$, in contrast with the elongated triangular and snub hexagonal grids. Our attempts suggest that discharging rules of the kind used in this paper, in which the amount that a vertex $v$ receives from each codeword depends only on $N[v]\cap C$, are unlikely to give lower bounds larger than $1/3$ for the rhombitrihexagonal grid and $3/10$ for the snub square grid, so new ideas seem to be needed for these two grids.

\section*{Acknowledgments}

The author used the AI assistant Claude to draw the figures, and to check all proofs. Its computer searches produced some codes of this paper; and it generated the certificate of Section~\ref{sec:snub} and the program that verifies it~\cite{rudini26idc}. The research directions, the structure of the paper, the writing and style of the proofs and the notation are due to the author, who checked all results and takes full responsibility for the paper.

\bibliographystyle{amsplain}
\bibliography{refs}

\end{document}